\documentclass[aap]{imsart}
\usepackage{newtxtext}
\usepackage{fix-cm}
\usepackage{latexsym,upgreek}
\usepackage{amsfonts}
\usepackage{amsmath}
\usepackage{amsthm}
\usepackage{graphicx} 
\usepackage{amssymb}
\usepackage{stmaryrd}
\usepackage{exscale}
\usepackage[square,numbers]{natbib}
\usepackage{bm}
\usepackage{bbm}
\usepackage{commath}
\usepackage{lipsum}
\usepackage{float} 
\usepackage{enumitem}
\usepackage{comment} 
\usepackage{cancel}
 \usepackage{tocloft}
 \usepackage{booktabs}

\graphicspath{{Fig/}}

\RequirePackage{amsthm,amsmath,amsfonts,amssymb}
\RequirePackage[colorlinks,citecolor=blue,urlcolor=blue]{hyperref}
\RequirePackage{graphicx}

\startlocaldefs

\newcommand{\mB}{{\mathcal B}}
\newcommand{\mP}{{\mathcal P}}
\newcommand{\mK}{{\mathcal K}}
\newcommand{\mC}{{\mathcal C}}
\newcommand{\mD}{{\mathcal D}}

\newcommand{\mI}{{\mathcal I}}

\newcommand{\mG}{{\mathcal G}}

\newcommand{\mX}{{\mathcal X}}
\newcommand{\mY}{{\mathcal Y}}
\newcommand{\mN}{{\mathcal N}}
\newcommand{\mM}{{\mathcal M}}

\newcommand{\mT}{{\mathcal T}}

\newcommand{\mbM}{\mathbb{M}}
\newcommand{\mbN}{\mathbb{N}}

\newcommand{\mbP}{\mathbb{P}}
\newcommand{\mbR}{\mathbb{R}}
 
\newcommand{\mbE}{\mathbb{E}} 

\newcommand{\sd}{{\sf d}}

\newcommand{\beq}{\begin{equation}}
\newcommand{\eeq}{\end{equation}}
\newcommand{\beqa}{\begin{eqnarray}}
\newcommand{\eeqa}{\end{eqnarray}}
\newcommand{\nn}{\nonumber}

\usepackage[most]{tcolorbox}

\newtcolorbox{codexbox}{
  breakable,
  colback=orange!10,
  colframe=orange!30!red,
  boxrule=0.4pt,
  arc=1mm,
  left=1mm,
  right=1mm,
  top=1mm,
  bottom=1mm,
  fonttitle=\bfseries,
  title=codex
}

\newtcolorbox{denizbox}{
  breakable,
  colback=red!10,
  colframe=red!60!black,
  boxrule=0.4pt,
  arc=1mm,
  left=1mm,
  right=1mm,
  top=1mm,
  bottom=1mm,
  fonttitle=\bfseries,
  title=\"{o}da
}

\theoremstyle{plain} 

\newtheorem{theorem}{Theorem}[section]
\newtheorem{lemma}[theorem]{Lemma}

\newtheorem{definition}[theorem]{Definition}
\newtheorem{proposition}[theorem]{Proposition}
\newtheorem{corollary}[theorem]{Corollary}

\newtheorem{remark}[theorem]{Remark}

\endlocaldefs

\begin{document}

\begin{frontmatter}
\title{An Operator-Theoretic Analysis of Nonlinear Filtering under Model Misspecification}
\runtitle{Nonlinear filtering under model misspecification}

\begin{aug}
\author[A]{\fnms{Fabián}~\snm{González} \thanks{[\textbf{Corresponding author.}]}\ead[label=e1]{omgonzal@math.uc3m.es}\orcid{0000-0003-4353-5737}}
\author[B]{\fnms{\"Omer Deniz}~\snm{Akyildiz}\ead[label=e2]{deniz.akyildiz@imperial.ac.uk}\orcid{0000-0002-5248-1219}}
\author[B]{\fnms{Dan}~\snm{Crisan}\ead[label=e3]{d.crisan@imperial.ac.uk}\orcid{0000-0003-1486-1517}}
\author[A]{\fnms{Joaquín}~\snm{Míguez}\ead[label=e4]{joaquin.miguez@uc3m.es}\orcid{0000-0001-5227-7253}}
\address[A]{Department of Signal Theory and Communications, Universidad Carlos III de Madrid, Spain\printead[presep={,\ } \\]{e1,e4}}
\address[B]{Department of Mathematics, Imperial College London, 180
Queen’s Gate, London SW7 2AZ, United Kingdom\printead[presep={,\ }]{e2,e3}}
\end{aug}

\begin{abstract}
We study nonlinear filtering with misspecified initial distributions and dynamics, but a correctly specified observation model. Observations are treated as a random process generated by the true model, rather than as fixed realizations, and both the correct filter and the misspecified filter are analyzed under this common observation law. Our main result is a Bayesian total variation balance identity for the expected change in total variation distance under a Bayesian update: a covariance term minus a non-negative remainder. The identity characterizes conditional and averaged update non-expansivity. Combined with prediction contraction, it yields finite-time bounds separating the initialization error from accumulated dynamical discrepancies, while allowing expansive updates. When the averaged updates are non-expansive and a uniform joint prediction--update coefficient is less than one, a persistent dynamical misspecification yields an asymptotic error bound proportional to the one-step discrepancy. When the conditional updates are non-expansive, we characterize an accumulated joint contraction that implies almost-sure and \(L^1\) coalescence of the true and misspecified filters. Finally, we verify update non-expansivity for scalar Gaussian predictive laws under a variance-ordering condition and, separately, derive explicit prediction contraction bounds for truncated kernels and bounded transition functions with elliptically symmetric noise.
\end{abstract}



\end{frontmatter}

%
\section{Introduction}
\label{sec:introduction}

%
\subsection{Motivation and problem statement}
\label{sec:problem-statement}

In a nonlinear filter, every prediction step with misspecified dynamics can introduce a new filtering error. Whether this error is damped or amplified depends on both the transition dynamics and the Bayesian update. This is a recurring difficulty in applications: the model used for inference may involve a numerical discretization, uncertain parameters or an incomplete description of the underlying physics, even when the observation mechanism is accurately known. The question is therefore not only whether the filter forgets an incorrect initial distribution (a classical stability problem \cite{del1999stability}), but whether it remains close to the optimal filter when new modelling errors are introduced over time.

We study this question for discrete-time state-space models (SSMs). An SSM describes an unobserved Markov state sequence \(\{X_t\}_{t\geq0}\) and observations \(\{Y_t\}_{t\geq1}\), whose conditional distributions are specified by likelihoods \(g_t(y_t\mid x_t)\) \cite{crisan2020stable}. Bayesian filtering computes the conditional law of \(X_t\) given the observations collected up to time \(t\) \cite{DelMoral96a,DelMoral04,Bain09}. We allow errors in the initial distribution and the transition kernels, while assuming that the observation model is correctly specified. This setting is natural when a sensor can be calibrated more accurately than the dynamics of the system it observes.

The forgetting of initial conditions, usually called \emph{filter stability}, is a classical problem in nonlinear filtering \cite{ocone1996asymptotic,del1999stability}; broader accounts are given in \cite{chigansky2006stability,chigansky2009intrinsic}. Quantitative stability
estimates also support the analysis of approximation schemes \cite{crisan2020stable}. Robustness to model perturbations has likewise been studied, including infinite-horizon estimates for continuous-time finite-state filters under mixing assumptions \cite{chiganskyvanHandel}. Dynamical misspecification presents a related difficulty: an initial discrepancy is introduced once, whereas a transition error can generate a fresh perturbation at every step. Its analysis must account for both the propagation of existing errors and the accumulation of new ones.

Several approaches seek to reduce model mismatch directly. Parameter estimation calibrates a family of SSMs to observations \cite{Andrieu10,Chopin12,Crisan18bernoulli,Perez-Vieites18,Perez-Vieites21}, robust filtering modifies the treatment of discrepancies between model and data \cite{calvet2015robust}, and nudging introduces observation-dependent corrections into the dynamics
\cite{akyildiz2020nudging,gonzalez2025nudging}. Our objective is complementary: for a given misspecified model, we quantify its filtering error relative to the true model.

Let \(K_t\) and \(\widehat K_t\) denote the true and misspecified prediction operators, and let \(U_{Y_t}\) be their common Bayesian update. Starting from priors \(\pi_0\) and \(\widehat\pi_0\), the two filters satisfy,
\[
\pi_t=U_{Y_t}\bigl(K_t(\pi_{t-1})\bigr)
\quad \text{and} \quad
\widehat\pi_t=U_{Y_t}\bigl(\widehat K_t(\widehat\pi_{t-1})\bigr),
\]
respectively. Both assimilate the same observations, generated by the true model. The filters are therefore random probability measures \cite{baccelli2024random}, and all expectations and almost-sure statements refer to the true observation law. We measure their discrepancy by the total variation distance \(\|\pi_t-\widehat\pi_t\|_{\mathrm{TV}}\), focusing on its expected magnitude and long-time behaviour. For readability, we use \(\pi_t\) here for the random filter denoted by \(\pi_{Y_{1:t}}\) in the main analysis.

Total variation controls the difference between the probabilities that the two filters assign to any measurable event, as well as discrepancies in expectations of bounded test functions. The resulting estimates therefore apply to a broad class of posterior summaries. Comparing exact filtering distributions also separates model mismatch from additional errors introduced by numerical approximation \cite{crisan2020stable,akyildiz2024sequential}.

The principal difficulty lies in the update. Prediction is linear and cannot increase total variation, whereas Bayesian updating is nonlinear: its normalizing constant depends on the predictive distribution and the new observation through its likelihood. Applying the same likelihood to two predictive distributions can bring them closer together or move them further apart. An update can even increase their distance on average under the true observation law, including in scalar Gaussian examples. Averaged Dobrushin estimates control this amplification by combining an update bound with prediction contraction \cite{mcdonald2020exponential}. The update factor in such a bound can be as large as two. Although useful, this bound does not capture the contraction actually supplied by an informative update. Understanding that contribution is central to our analysis.

%
\subsection{Contributions}

We formulate filtering as a sequence of random operators acting on probability measures. The mean optimal filter follows a linear recursion driven by prediction, but controlling the expected distance between filters also requires understanding how updating changes that distance. Our main contributions concern this update effect and its consequences for error propagation.

%
\subsubsection{An exact balance identity for Bayesian updating}

Let \(\xi_t=K_t(\pi_{t-1})\) be the true predictive distribution, let \(\nu_t\) be an alternative predictive distribution determined by the same past observations $Y_{1:t-1}$, and write \(\mathcal G_{t-1}=\sigma(Y_{1:t-1})\) for the sigma-algebra generated by $Y_{1:t-1}$. Under the likelihood assumptions of Section~\ref{sec2}, we establish, almost surely (a.s.), the \emph{Bayesian total variation balance identity}:
\beq
\mbE\left[
	\|U_{Y_t}(\xi_t)-U_{Y_t}(\nu_t)\|_{\mathrm{TV}}
	\middle| 
	\mathcal G_{t-1}
\right]
-\|\xi_t-\nu_t\|_{\mathrm{TV}}
=\mathcal T_{{\rm cov},t}-\mathcal T_{{\rm rem},t},
~~\text{ where }~~ \mathcal T_{{\rm rem},t}\geq0.
\label{eq-key1}
\eeq
Both terms in \eqref{eq-key1} are evaluated at \((\xi_t,\nu_t)\). The covariance term $\mathcal T_{{\rm cov},t}$ links relative predictive likelihoods with a posterior probability; the nonnegative remainder $\mathcal T_{{\rm rem},t}$ accounts for regions where the sign of the predictive difference changes after updating. The balance characterizes non-expansivity of the update operator $U_{Y_t}$ exactly, conditionally on the past or after averaging over observation histories. In particular, a non-positive covariance term is sufficient, but a positive covariance may also be offset by the remainder. The identity does not require absolute continuity between the predictive distributions.

%
\subsubsection{Error propagation through prediction and updating}

We use the identity \eqref{eq-key1} to quantify how much non-expansive updates reduce the total variation discrepancy between the predictive distributions. The corresponding contraction rates are \(\delta_t\), conditional on the observation history, and \(\bar\delta_t\), after averaging over observation histories. Both rates depend on the predictive pair being compared. Combining them with prediction contraction gives finite-time bounds that separate the propagated initial error from newly introduced dynamical discrepancies. Each discrepancy can be damped or amplified by the update at the time it enters and by subsequent steps. The bounds also accommodate potentially expansive updates. For an alternative initialization of the true filtering recursion, convergence follows when accumulated contraction from prediction and updating dominates this amplification. Verifying the criterion for every alternative initialization gives stability of the filtering operators.

When only the initial distribution is misspecified and both filters use the true transition kernels ($K_t=\widehat K_t$), we also characterize stability in expectation through convergence in probability: for every alternative initial distribution, the probability that the filtering error exceeds any fixed tolerance tends to zero. Under conditional update non-expansivity, the error process is a nonnegative supermartingale, linking the operator bounds with probabilistic convergence arguments.

%
\subsubsection{Long-time bounds under misspecified dynamics}

Consider again the case where $K_t \ne \widehat K_t$. Two regimes illustrate the results of our analysis. If every update is non-expansive in expectation and the joint coefficient \(C:=\sup_{t\geq1}\|K_t\|(1-\bar\delta_t)\) is strictly below one, the bound contains an exponentially decaying contribution from the initial error. In addition, suppose that, starting from any common distribution of the state at time \(t-1\), the true and misspecified dynamics produce predictive state distributions at time \(t\) whose total variation distance is at most \(\varepsilon\), uniformly in \(t\). Then,
\[
\limsup_{t\to\infty}
\mbE\!\left[\|\pi_t-\widehat\pi_t\|_{\mathrm{TV}}\right]
\leq\frac{\varepsilon}{1-C}.
\]
Persistent model errors thus yield an asymptotic bound proportional to their magnitude. The joint coefficient makes explicit that observations can supply contraction even when prediction alone does not. The estimate controls the long-run discrepancy without requiring the transition error to disappear. It improves on the universal total variation bound of one when \(\varepsilon<1-C\), and shows how stronger joint contraction reduces the effect of repeatedly introduced errors.

We also prove that the true and misspecified filters coalesce a.s. and in \(L^1\) when updates are conditionally non-expansive, the effective dynamical forcing is summable and joint contraction accumulates without bound, i.e.,
$$
\sum_{t\ge1}(1-\delta_t)\varepsilon_{t-1}<\infty
\quad \text{and} \quad
\sum_{t\ge1}\bigl[1-(1-\delta_t)\|K_t\|\bigr]=\infty \quad\text{a.s.},
$$
where \(\varepsilon_{t-1}\) measures the prediction discrepancy introduced at time \(t\), and \(\delta_t\) is the conditional contraction rate for the true and misspecified predictive pair. Neither prediction nor updating needs a uniformly positive contraction rate. Raw summability of the prediction discrepancies $\delta_t$ is sufficient, but effective summability also permits updates to suppress persistent discrepancies. The proof uses the Robbins--Siegmund almost-supermartingale theorem \cite{robbins1971convergence}. An analogous formulation uses trajectory-dependent discrepancies and the contraction coefficients of the misspecified prediction operators.

%
\subsubsection{Analytical examples}
For scalar Gaussian predictive distributions, an alternative variance at least as large as the true variance ensures conditional update non-expansivity without restrictions on the means. This ordering is preserved by the scalar Kalman recursion when both filters use the same dynamics and observation model. Separately, Doeblin minorization yields explicit prediction contraction estimates for positive continuous transition densities on compact state spaces and bounded transition maps with suitable elliptically symmetric noise, including Gaussian, Laplace and Student-\(t\) examples. These examples verify concrete update and prediction mechanisms underlying the general results.

%
\subsection{Outline of the paper}

This section concludes with a brief summary of common notation used throughout the paper. Section~\ref{sec2} introduces the filtering model, random measures and the operator formulation. Section~\ref{sec3} establishes the update balance identity and develops the stability analysis. Section~\ref{sec4} derives error bounds and coalescence results under dynamical misspecification. Section~\ref{sec5} presents the analytical examples, and Section~\ref{sec6} concludes. Technical proofs and additional results are collected in the
appendices.

%
\subsection{Notation}
\label{ssNotation}

For a Borel set \(S\subseteq\mbR^d\), \(\mathcal B(S)\) denotes its Borel
\(\sigma\)-algebra, \(\mathcal P(S)\) the probability measures on it, and
\(\mbM(S)\) the finite signed measures. For a measure \(m\) and an
integrable function \(f\), we write \(m(f):=\int f(s)\,m(\mathrm ds)\).
Our total variation convention is
\[
\|\mu-\nu\|_{\mathrm{TV}}
:=\sup_{A\in\mathcal B(S)}|\mu(A)-\nu(A)|
=\tfrac12|\mu-\nu|(S)\in[0,1],
\qquad \mu,\nu\in\mathcal P(S),
\]
where \(|m|\) is the total variation measure of \(m\). More generally,
\(\|m\|_{\mathrm{TV}}=\tfrac12|m|(S)\) when \(m(S)=0\).
The same symbol \(K_t\) denotes a transition kernel and its prediction
operator; \(\|K_t\|\leq1\) is its induced contraction coefficient in
total variation.

We denote by \(B(S)\) the bounded measurable real functions on \(S\),
with \(\|f\|_\infty:=\sup_{s\in S}|f(s)|\). Subscript strings denote
sequences, for example \(Y_{1:t}=(Y_1,\ldots,Y_t)\); the same convention
applies to subsequences \(x_{t_1:t_n}\). Capital letters denote random
variables and lowercase letters their realizations. We write
\(\mathcal N(\mathrm dx;\mu,\Sigma)\) for a Gaussian law and
\(\mathcal N(x;\mu,\Sigma)\) for its density.

Expectation and covariance under a specified law \(\rho\) are denoted
by \(\mbE^\rho\) and \(\operatorname{Cov}^\rho\); without a superscript,
expectation is under the true model. The mean of a random measure is
\(M_{\pi_t}:=\mbE[\pi_t]\). We use the usual spaces \(L^p(\rho)\),
with \(\|f\|_{L^p(\rho)}=(\rho(|f|^p))^{1/p}\) for \(1\leq p<\infty\).
The symbols \(\xrightarrow{\mathbb P}\), \(\xrightarrow{\mathrm{a.s.}}\)
and \(\xrightarrow{L^p}\) denote convergence in probability, almost
surely and in \(L^p(\mathbb P)\), respectively.
For a real symmetric matrix \(A\), \(\mathrm{spec}(A)\) denotes its
spectrum, and \(\lambda_{\min}(A)\), \(\lambda_{\max}(A)\) its extreme
eigenvalues.


%
\section{Filters with random observations}\label{sec2}

In this section, we introduce the random observation formulation of SSMs, which serves as the basis for the analysis developed throughout the paper. In contrast to most of the literature on the problem, where observations are treated as fixed realizations, we consider the observation process as a stochastic object defined on the same probability space as the state (see, e.g., \cite{mcdonald2020exponential}). This perspective allows us to study the filtering recursion as a sequence of random probability measures and to exploit the probabilistic structure of the observations in the analysis of this sequence. In particular, it enables us to characterize the evolution of the filter in expectation and to relate key properties, such as stability and error propagation under model misspecification, to the contraction of the underlying operators.

\subsection{Model}
Let $(\Omega, \mathcal F, \mathbb P)$ be a probability space. We consider two discrete-time stochastic processes indexed by $t \in \mathbb N$:
\begin{itemize}
    \item the \emph{state} (or \emph{signal}) process $X = \{X_t\}_{t \geq 0}$ taking values in $\mathcal X \subseteq \mathbb R^{d_x}$;
    \item the \emph{observation} process $Y = \{Y_t\}_{t \geq 1}$ taking values in $\mathcal Y \subseteq \mathbb R^{d_y}$.
\end{itemize}
We refer to $\mathcal X$ and $\mathcal Y$ as the state space and the observation space, respectively. Both are assumed to be Borel subsets of their respective Euclidean spaces. The initial state $X_0$ is distributed according to the prior probability law
\begin{equation*}
    \pi_0(\mathrm dx) := \mathbb P(X_0 \in \mathrm dx).
\end{equation*}
After time $t=0,$ the evolution of the state process is governed by a sequence of Markov transition kernels $\{K_t\}_{t \geq 1}$, defined as
\[
K_t(x_{t-1}, A) := \mathbb P(X_t \in A \mid X_{t-1} = x_{t-1}),
\qquad A \in \mathcal B(\mathcal X),
\]
where $\mathcal B(\mathcal X)$ denotes the Borel $\sigma$-algebra on $\mathcal X$.
Conditionally on $X_t = x_t$, the observation $Y_t$ is assumed to admit a density $g_t(y_t \mid x_t)$ w.r.t. a reference measure on $\mathcal Y$ (typically the Lebesgue measure). For each fixed observation $y_t \in \mathcal Y$, we write
\[
g_{y_t}(x_t) := g_t(y_t \mid x_t),
\]
thus viewing $g_{y_t}(\cdot)$ as a function of the state variable (i.e., a likelihood). The observations are assumed to be conditionally independent given the state process.

The triplet $\mathcal M := (\pi_0, K, g)$, where $K = \{K_t\}_{t \geq 1}$ and $g = \{g_t\}_{t \geq 1}$, defines a discrete-time SSM \cite{crisan2020stable}. Given a realization of the observation sequence  $Y_{1:t} = y_{1:t}$, the deterministic filtering distribution at time $t$ is
\[
\pi_t(\mathrm dx) := \mathbb P(X_t \in \mathrm dx \mid Y_{1:t} = y_{1:t}), \qquad t \geq 1,
\]
and the sequence $\{\pi_t\}_{t \geq 1}$ satisfies a classical prediction--update recursion \cite{Bain09}. 
\begin{itemize}
\item   
The one-step-ahead predictive law of $X_t$ is
\[
\xi_t(\mathrm dx)
:= \mathbb P(X_t \in \mathrm dx \mid Y_{1:t-1} = y_{1:t-1})
= \int_{\mathcal X} K_t(x', \mathrm dx)\,\pi_{t-1}(\mathrm dx'),
\]
which we also write compactly as $\xi_t = K_t(\pi_{t-1})$.

\item Given a newly collected observation $Y_t=y_t$, the filtering measure \cite{stone2023introduction,Robert04} at time $t$ is 
\[\pi_t(\sd x)= \frac{g_{y_t}(x) \xi_t(\sd x)}{\xi_t(g_{y_t})}\]
and the integral of a test function $f$ w.r.t. $\pi_t$ can be computed as
\begin{equation*}
    \pi_t(f)
    = \frac{\xi_t(f g_{y_t})}{\xi_t(g_{y_t})}, \quad \text{where} \quad  \xi_t(fg_{y_t}) := \int_{\mathcal X} f(x) g_{y_t}(x)\,\xi_t(\mathrm dx).
\end{equation*}
\end{itemize}

\begin{remark}\label{rem:likelihood_assumptions}
We assume that the likelihood functions
\(
g_t : \mathcal Y \times \mathcal X \to (0,\infty)
\)
are jointly measurable, strictly positive and normalized w.r.t. the Lebesgue measure, i.e., $g_t>0$ and
\begin{equation}\label{eq:g_normalized}
    \int_{\mathcal Y} g_t(\mathrm{y} \mid x_t)\,\sd \mathrm{y} = 1,
    \qquad \forall x_t \in \mathcal X,\;\; t \in \mathbb N.
\end{equation}
Under this assumption, $g_t(y \mid x_t)$ coincides with the conditional pdf of $Y_t$ given $X_t = x_t$.
In addition, we assume that each likelihood is bounded, i.e.,
\[
\|g_{y_t}\|_\infty
:=\sup_{x\in\mathcal X}g_t(y_t\mid x)<\infty,
\qquad t\geq1,\quad y_t\in\mathcal Y.
\]
The bound needs not be uniform, it may depend on \(t\) and \(y_t\).
\end{remark}

\subsection{Random observations}
Hereafter, let us assume that the observation sequence $\{Y_t\}_{t \geq 1}$ is random and generated by the model itself,
\beq
\begin{array}{llll}
    Y_t: &\Omega &\to &\mY\\
    &\omega &\mapsto &Y_t(\omega). \notag
\end{array}
\eeq
For each $t$, we define the random likelihood
\(g_{Y_t}(x_t):=g_t(Y_t \mid x_t),\)
which is itself a r.v. for every $x_t.$ Within this framework, the filtering distributions are random probability measures (see \cite{baccelli2024random}). Throughout, \(\mathcal P(\mathcal X)\) is equipped with the \(\sigma\)-field generated by the evaluation maps \(\rho\mapsto\rho(C)\), \(C\in\mathcal B(\mathcal X)\). The random filter at time $t$ becomes
\begin{equation}\label{eq:random_filter}
\pi_{Y_{1:t}}(\sd x)
= \frac{
g_{Y_t}(x)\xi_{Y_{1:t-1}}(\sd x)
}{
\xi_{Y_{1:t-1}}(g_{Y_t})
},
\end{equation}
where the predictive measure is
\begin{equation*}
\xi_{Y_{1:t-1}}
= K_t(\pi_{Y_{1:t-1}}).
\end{equation*}
We refer to \(\{\pi_{Y_{1:t}}\}_{t \geq 1}\) as the sequence of random optimal filters \cite{Bain09}, where the conditional law is computed under the true model $\mM$. Both \(\pi_{Y_{1:t}}\) and \(\xi_{Y_{1:t}}\) are now random measures and, for a bounded measurable test function \(f:\mathcal X \to \mathbb R\), the quantity
\[
\pi_{Y_{1:t}}(f)
=
\frac{\xi_{Y_{1:t-1}}(f g_{Y_t})}
{\xi_{Y_{1:t-1}}(g_{Y_t})}
\]
is a r.v..

Since our analysis is carried out in a random-observation setting, it becomes necessary to explicitly characterize the probabilistic structure induced by the observation process itself. In particular, we require an explicit representation of the conditional law of \(Y_t\) given the past observations \(Y_{1:t-1}\). The following remark shows that the normalization constant appearing in the Bayes update coincides with this conditional density. This identification will play a fundamental role throughout the paper.
\begin{remark}\label{remark:likelihood_identity}
For each fixed observation sequence \(Y_{1:t-1}=y_{1:t-1}\), the function
\(
\mathrm{y_t}
\mapsto
\xi_{y_{1:t-1}}(g_{\mathrm{y_t}})
\)
coincides with the conditional pdf of \(Y_t\) given
\(Y_{1:t-1}=y_{1:t-1}\), which we denote as
\(
p_{Y_t\mid y_{1:t-1}}(\mathrm{y_t})
=
\xi_{y_{1:t-1}}(g_{\mathrm{y_t}}).
\)
Consequently, evaluating this identity at the random observation history \(Y_{1:t-1}\) yields
\[
\xi_{Y_{1:t-1}}(g_{\mathrm{y_t}})
=
p_{Y_t\mid Y_{1:t-1}}(\mathrm{y_t})
\qquad \text{a.s.,}
\]
i.e., \(\xi_{Y_{1:t-1}}(g_{\mathrm{y_t}})\) is the conditional density of \(Y_t\) given the observation filtration \(\sigma(Y_{1:t-1})\).
\end{remark}

\begin{proposition}\label{prop:Ypdf}
Let \(\mathcal G_{t-1} := \sigma(Y_{1:t-1})\) be the \(\sigma\)-algebra generated by the random observations \(Y_{1:t-1}\). Then, for any integrable function \(f : \mathcal Y^t \to \mathbb R\),
\begin{equation*}
\mathbb E\!\left[f(Y_{1:t}) \mid \mathcal G_{t-1}\right]
=
\int_{\mathcal Y}
f(Y_{1:t-1}, \mathrm{y_t})\,
\xi_{Y_{1:t-1}}(g_{\mathrm{y_t}})\,
\mathrm d \mathrm{y_t}.
\end{equation*}
\end{proposition}
The proof can be found in Appendix~\ref{proof:Ypdf}.


\subsection{Random optimal filters}

Our goal is to analyze the quantity
\[
\pi_{Y_{1:t}}(f) = \int_{\mathcal X} f(x)\, \pi_{Y_{1:t}}(\mathrm dx),
\]
that is, the expectation of a test function \( f \in B(\mathcal X) \) w.r.t. the random filtering measure \( \pi_{Y_{1:t}} \), for \( t \ge 1 \). Since the filter depends on the observation sequence \( Y_{1:t} \), the quantity \( \pi_{Y_{1:t}}(f) \) is itself a r.v..
We first recall some basic notions from the theory of random measures; see \cite{baccelli2024random} for a general discussion.

\begin{definition}
Let \((\mathcal X,\mathcal B(\mathcal X))\) be a measurable space and let
\(\mu_Y\) be a random measure on \(\mathcal X\).
Its mean measure is defined by
\(
M_{\mu_Y}(B)
:=
\mathbb E[\mu_Y(B)],
\;
B\in\mathcal B(\mathcal X).
\)
\end{definition}
The following fundamental result allows one to exchange expectation and integration w.r.t. a random measure.

\begin{theorem}[Campbell averaging formula]\label{thm:CAF}
Let \( \mu_Y \) be a random measure on a measurable space \( (\mathcal X,\mB(\mX)) \), with $\mX \subseteq \mathbb R^{d}$, $d\in \mbN$, and let \( M_{\mu_Y} \) denote its mean measure. Then, for any measurable function \( f : \mathcal X \to \overline{\mathbb R} \), where \(\overline{\mathbb R}\) denotes the extended real line, such that \(f\) is \(M_{\mu_Y}\)-integrable, 
\[
\mathbb E\!\left[\int_{\mathcal X} f(x)\, \mu_Y(\mathrm dx) \right]
=
\int_{\mathcal X} f(x)\, M_{\mu_Y}(\mathrm dx).
\]
\end{theorem}

\begin{proof}
See Theorem~1.2.5 in \cite{baccelli2024random}.
\end{proof}
By Theorem~\ref{thm:CAF}, we obtain
\(
\mathbb E\!\left[\pi_{Y_{1:t}}(f)\right]
= \int_{\mathcal X} f(x)\, M_{\pi_{Y_{1:t}}}(\mathrm dx),
\)
so we aim to characterize the mean measure
\(
M_{\pi_{Y_{1:t}}}=\mathbb E[\pi_{Y_{1:t}}].
\) \begin{corollary}\label{thm:mean_filter}
Let \(\{\pi_{Y_{1:t}}\}_{t\geq 1}\) denote the random filtering sequence defined in
\eqref{eq:random_filter}, and let \(M_t := \mathbb E[\pi_{Y_{1:t}}]\) denote its expected (mean) measure. Then \(M_t\) is the unconditioned distribution of the state \(X_t\) and satisfies the linear recursion
\beq
M_t = K_t(M_{t-1}), \quad \text{for all } t \ge 1,
\label{eq-extra-01}
\eeq
with initial condition \(M_0 = \pi_0\). Equivalently, for any bounded measurable test function \(f:\mathcal X \to \mathbb R\),
$
M_t(f) = M_{t-1}(K_t f).
$
\end{corollary}

\begin{proof}
Let \(\mathcal G_t = \sigma(Y_{1:t})\). By the tower property of conditional expectation, for any bounded measurable test function \(f\),
\[
M_t(f) = \mathbb E\left[ \pi_{Y_{1:t}}(f) \right] 
= \mathbb E\left[ \mathbb E[f(X_t) \mid \mathcal G_t] \right] 
= \mathbb E[f(X_t)],
\]
hence $M_t=\text{Law}(X_t)$. Moreover, $\text{Law}(X_t)=K_t\left( \text{Law}(X_{t-1}) \right)$, which yields \eqref{eq-extra-01}.
\end{proof}

%
\subsection{An operator-theoretic filtering model}

%
\subsubsection{Prediction and update operators}
Let \(\mbM(\mX)\) denote the space of totally finite signed measures on
\((\mX,\mB(\mX))\) \cite{folland1999real}. The operator-theoretic formulation of the filtering recursion is based on the following prediction and update operators.

The prediction operator $K_t : \mbM(\mX) \to \mbM(\mX)$ is defined by
\begin{equation}\label{PO}
    K_t(\mu)(A)
    := \int_{\mX} K_t(x, A)\, \mu(\mathrm dx),
    \qquad A \in \mB(\mX).
\end{equation}
In particular, the predictive measure associated with the filtering recursion is given by
\(
\xi_{Y_{1:t-1}}
=
K_t(\pi_{Y_{1:t-1}}).
\)

The update operator $U_t : \mP(\mX) \times \mY \to \mP(\mX)$ is defined by
\beq\label{EUO}
U_t(\mu, y_t) := \frac{g_{y_t}(x)\,\mu}{\mu(g_{y_t})}.
\eeq
For conciseness, we write $U_{y_t}(\mu) := U_t(\mu, y_t)$.
By the strict positivity and boundedness assumptions in Remark~\ref{rem:likelihood_assumptions},
\[
0<\mu(g_{y_t})\leq\|g_{y_t}\|_\infty<\infty
\]
for every \(\mu\in\mathcal P(\mathcal X)\), \(y_t\in\mathcal Y\)
and \(t\geq1\). Hence the update is well defined on all of
\(\mathcal P(\mathcal X)\times\mathcal Y\).
By composition, we define the non-linear \emph{random prediction--update operator} 
\begin{equation}\label{PEUO}
    \Phi_{Y_t} := U_{Y_t} \circ K_t.
\end{equation}
For any \(\mu \in \mathcal P(\mathcal X)\), the mapping
\(
\Phi_{Y_t}(\mu)
\)
defines a random probability measure, depending on the input measure \(\mu\) and the random observation \(Y_t\). More generally, for $1 \leq k < t$, we define the composition
\begin{equation}\label{RPEU}
\Phi_{Y_{k:t}}
:= \Phi_{Y_t} \circ \Phi_{Y_{t-1}} \circ \cdots \circ \Phi_{Y_k}.
\end{equation}
With this notation, the filtering recursion can be written compactly as
\[
\pi_{Y_{1:t}} = \Phi_{Y_{t}}(\pi_{Y_{1:t-1}})=U_{Y_t} \circ K_t(\pi_{Y_{1:t-1}})=\Phi_{Y_{1:t}}(\pi_0), \quad t\ge1.
\]
Note that $\Phi_{Y_{k:t}}$ depends only on the kernels $K_i$ and likelihoods $g_{Y_i}$ for $i=k,\dots,t$, and not on the initial measure $\pi_0$.\footnote{We use the convention \(\Phi_{Y_{1:0}}=\widehat\Phi_{Y_{1:0}}=I\), where \(I\) is the identity operator, and identify the time-zero filtering measures with their priors: \(\pi_{Y_{1:0}}=\pi_0\) and \(\widehat\pi_{Y_{1:0}}=\widehat\pi_0\). The hatted notation refers to the misspecified recursion introduced in Section~\ref{sec4}.}

\begin{definition}
Define
\(
\mD(\mX)
:=
\left\{
\beta(\pi_1-\pi_2)
:\;
\beta\in\mathbb R,\;
\pi_1,\pi_2\in\mathcal P(\mX)
\right\}.
\)
Then \(\mD(\mX)\) is the linear subspace of \(\mbM(\mX)\) generated by the difference of probability measures.
\end{definition}


\begin{remark}
We note that \((\mD(\mX),\|\cdot\|_{\mathrm{TV}})\) is a Banach space. In addition, the predictive operators \(K_t\) act linearly on \(\mD(\mX)\) and each \(K_t\) defines a bounded linear operator on \(\mD(\mX)\) and its induced operator norm is
\[
\|K_t\|
:=
\sup_{\substack{\lambda\in\mD(\mX)\\ \lambda\neq0}}
\frac{\|K_t(\lambda)\|_{\mathrm{TV}}}
{\|\lambda\|_{\mathrm{TV}}}.
\]
Furthermore,
\(
\|K_t\|\leq1.
\)
and, therefore, \(K_t\) is non-expansive on \(\mD(\mX)\). We reserve the term ``contractive'' for kernels $K$ such that $\| K \| < 1$ (note the strict inequality).
\end{remark}

As a direct consequence of the linearity of the prediction operator $K_t$, for any $\mu,\nu \in \mathcal P(\mX)$,

\beq\label{ineq:Psi_bound}
\|K_t(\mu) - K_t(\nu)\|_{\mathrm{TV}}
\leq \|K_t\|\, \|\mu - \nu\|_{\mathrm{TV}}.
\eeq
We now provide an upper bound for the update operator $U_t$.

\begin{proposition}\label{DUI}
Let $\mu, \nu \in \mathcal P(\mathcal X)$. Then, for every
\(y_t\in\mathcal Y\),
\[
\|U_{y_t}(\mu) - U_{y_t}(\nu)\|_{\mathrm{TV}}
\leq \frac{
	\int_{\mathcal X} g_{y_t}(x)\, |\mu - \nu|(\mathrm dx) + |\nu(g_{y_t}) - \mu(g_{y_t})|
}{
	2\max \big\{ \mu(g_{y_t}), \nu(g_{y_t}) \big\}
}
.
\]
\end{proposition}
A proof is provided in Appendix~\ref{proof:DUI}

%
\subsubsection{The Link operator}
We introduce the \emph{link operator} $L_t : \mbM(\mX) \to \mbM(\mY)$ defined by
\begin{equation}
L_t(\mu)(F)
: = \int_F \int_{\mX} g_{\mathrm{y_t}}(x) \mu(\sd x) \sd \mathrm{y_t}= \int_F \mu(g_{\mathrm{y_t}})\, \mathrm d\mathrm{y_t},
\qquad F \in \mathcal B(\mathcal Y).
\end{equation}
This operator maps a measure on the state space into a measure on the observation space. By construction, $L_t$ is linear and if $\mu \in \mathcal P(\mathcal X)$, then $L_t(\mu) \in \mathcal P(\mathcal Y)$. Indeed,
\[
L_t(\mu)(\mathcal Y)
= \int_{\mathcal X} \left( \int_{\mathcal Y} g_{\mathrm{y_t}}(x)\, \mathrm d\mathrm{y_t} \right) \mu(\mathrm dx)
= \mu(\mathcal X) = 1.
\]
Most importantly, for our analysis,  $L_t$ maps $\mD(\mX)$ into $\mD(\mY)$. Indeed, if $\lambda = \beta(\mu - \nu)$ with $\mu,\nu \in \mathcal P(\mathcal X)$, then by linearity
\[
L_t(\lambda) = \beta \big(L_t(\mu) - L_t(\nu)\big).
\]
\begin{proposition}\label{prop:Link_Op}
The link operator $L_t$ is a bounded linear operator from $(\mathcal D(\mathcal X),\|\cdot\|_{\mathrm{TV}})$ to $(\mathcal D(\mathcal Y),\|\cdot\|_{\mathrm{TV}})$. Moreover, $\|L_t\|\leq1$ for all $t\geq1$.
\end{proposition}
A proof is provided in Appendix~\ref{proof:Link_Op}.
\begin{remark}\label{rem:UBT_LP}
Denote by
\[
0\leq \sup_{t\geq1}\|L_t\|\leq C_L\leq 1,
\qquad
0\leq \sup_{t\geq1}\|K_t\|\leq C_K\leq 1,
\]
the uniform upper bounds on the operator norms.
\end{remark}

\begin{proposition}\label{prop:normalizing_constant_positive}
Assume that the likelihoods $g_t$ are as described in Remark~\ref{rem:likelihood_assumptions} and let \(\mu,\nu\in\mathcal P(\mathcal X)\) be arbitrary probability measures. Then, for every \(t\geq1\) and \(y_t\in\mathcal Y\),
\[
0<\mu(g_{y_t})<\infty
\quad\text{and}\quad
0<\nu(g_{y_t})<\infty.
\]
In particular, both normalizing constants may be used in the Bayes updates for every observation, without any absolute continuity assumption relating \(\mu\) and \(\nu\).
\end{proposition}

A proof is provided in Appendix~\ref{proof:normalizing_constant_positive}.

%
\section{Stability}\label{sec3}

\subsection{Stability of the random optimal filter}

In this section, we formalize the notion of stability for the sequence of
random filtering operators and characterize conditions under which the
filter asymptotically forgets its initial condition. Unlike the classical
formulation, where stability is studied along a fixed observation path, we
work within the random-observation framework of Section~\ref{sec2}.
Consequently, the filtering measures generated by the random operators
\(\{\Phi_{Y_t}\}_{t\geq1}\) are themselves random probability measures, and
stability is analyzed in expectation. Within this framework, stability
results from the interplay between the contraction induced by the predictive
operators \(K_t\) and the possible expansion produced by the update
operators \(U_t\). This operator-theoretic characterization quantifies how
perturbations in the initial distribution propagate through time and
provides the basis for the misspecification analysis in Section~\ref{sec4}.

Following the notion of filter stability in Section~1.2 of
\cite{crisan2020stable}, we say that the sequence of random filtering
operators generated by a SSM is \emph{stable} if the dependence of the
filtering recursion on its initial distribution vanishes as \(t\to\infty\).

\begin{definition}\label{def:Ex_St}
Let \(\{\Phi_{Y_t}\}_{t\geq1}\) denote the sequence of random
prediction--update operators defined in \eqref{PEUO} and \eqref{RPEU} on
\(\mathcal P(\mathcal X)\). The sequence
\(\{\Phi_{Y_t}\}_{t\geq1}\) is said to be stable in expectation if, for
every pair of initial probability measures
\(\mu_0,\nu_0\in\mathcal P(\mathcal X)\),
\[
\lim_{t\to\infty}
\mathbb E\!\left[
\left\|
\Phi_{Y_{1:t}}(\mu_0)
-
\Phi_{Y_{1:t}}(\nu_0)
\right\|_{\mathrm{TV}}
\right]
=0.
\]
\end{definition}

\begin{remark}
The terminology ``stability of the optimal filter'' may be slightly
misleading. Indeed, as the previous definition shows, stability is
fundamentally a property of the sequence of random operators
\(\{\Phi_{Y_t}\}_{t\geq1}\). For this reason, throughout the sequel we will
often refer to the stability of the associated random operators.
\end{remark}

Given two initial probability measures
\(
\mu_0,\nu_0\in\mathcal P(\mathcal X),
\)
we denote the corresponding random filtering measures, for \(t\geq1\), as
\[
\mu_{Y_{1:t}}
:=
\Phi_{Y_t}(\mu_{Y_{1:t-1}})
=
\Phi_{Y_{1:t}}(\mu_0) 
\quad \text{and} \quad
\nu_{Y_{1:t}}
:=
\Phi_{Y_t}(\nu_{Y_{1:t-1}})
=
\Phi_{Y_{1:t}}(\nu_0).
\]
Recall that the corresponding mean measures are
\(
M_{\mu_{Y_{1:t}}}
=
\mathbb E[\mu_{Y_{1:t}}]
\)
and
\(
M_{\nu_{Y_{1:t}}}
=
\mathbb E[\nu_{Y_{1:t}}].
\)
Applying Jensen's inequality with the absolute-value function to \(\mu_{Y_{1:t}}(A)-\nu_{Y_{1:t}}(A)\) for each \(A\in\mathcal B(\mathcal X)\), and then taking the supremum over these sets, yields
\beq\label{MM_NC}
\|M_{\mu_{Y_{1:t}}}-M_{\nu_{Y_{1:t}}}\|_{\mathrm{TV}}
\leq
\mathbb E\!\left[
\|\mu_{Y_{1:t}}-\nu_{Y_{1:t}}\|_{\mathrm{TV}}
\right].
\eeq
Therefore, a necessary condition for the stability of the operators
\(\{\Phi_{Y_t}\}_{t\geq1}\) is the coalescence of the corresponding mean
measures. We now investigate sufficient conditions for stability.

We write
\(
\mathcal G_t:=\sigma(Y_{1:t})
\)
for the \(\sigma\)-algebra generated by the observations up to time \(t\), and let \(C_K\) denote the uniform bound of the operator sequence \(\{\|K_t\|\}_{t\geq1}\) introduced in Remark~\ref{rem:UBT_LP}. Unless otherwise stated, every random probability measure considered hereafter arises explicitly from the action of the filtering operators \(\{\Phi_{Y_t}\}_{t\geq1}\) on \(\mathcal P(\mathcal X)\), and is therefore
a function of the observation path. Whenever it is useful to emphasize this dependence, we write \(\nu_{Y_{1:t}}\). 

We now turn to the problem of characterizing the stability of the random
filtering operators \(\{\Phi_{Y_t}\}_{t\geq1}\). To this end, we follow the
approach of \cite{crisan2020stable,mcdonald2020exponential}, beginning with
a collection of preliminary estimates.

%
\subsection{Convergence rates}

We begin with an averaged stability estimate for the update operator \(U_t\).
The following lemma bounds the expected discrepancy between two updated
measures in terms of the discrepancy between their inputs.

\begin{lemma}\label{lem:Emu}
Let \(\mu,\nu\in\mathcal P(\mathcal X)\), and let
\(\mathbb E^{L_t(\mu)}\) denote expectation w.r.t. the probability
measure \(L_t(\mu)\in\mathcal P(\mathcal Y)\). Then
\[
\mathbb{E}^{L_t(\mu)}\!\left[
\|U_{Y_t}(\mu)-U_{Y_t}(\nu)\|_{\mathrm{TV}}
\right]
\leq
(1+\|L_t\|)\,
\|\mu-\nu\|_{\mathrm{TV}}.
\]
\end{lemma}

\begin{proof}
By Remark~\ref{rem:likelihood_assumptions}, the normalizers \(\mu(g_{y_t})\) and \(\nu(g_{y_t})\) are positive and finite for every \(y_t\in\mathcal Y\). Since \(L_t(\mu)(\mathrm d \mathrm y_t)=\mu(g_{\mathrm y_t})\,\mathrm d \mathrm y_t\),
Proposition~\ref{DUI} and
\(\mu(g_{y_t})/\max\{\mu(g_{y_t}),\nu(g_{y_t})\}\leq1\) yield
\beqa
\mathbb E^{L_t(\mu)}\!\left[
\|U_{Y_t}(\mu)-U_{Y_t}(\nu)\|_{\mathrm{TV}}
\right]
&\leq& \frac12\int_{\mathcal Y}
\left[|\mu-\nu|(g_{\mathrm y_t})+
\bigl|\nu(g_{\mathrm y_t})-\mu(g_{\mathrm y_t})\bigr|\right]\,\mathrm d \mathrm y_t \nn \\
&=& \frac12|\mu-\nu|(\mathcal X)
+\frac12\int_{\mathcal Y}
\bigl|\nu(g_{\mathrm y_t})-\mu(g_{\mathrm y_t})\bigr|\,\mathrm d \mathrm y_t \nn \\
&=& \|\mu-\nu\|_{\mathrm{TV}}
+\|L_t(\mu)-L_t(\nu)\|_{\mathrm{TV}} \nn \\
&\leq& (1+\|L_t\|)\,\|\mu-\nu\|_{\mathrm{TV}}.
\nn
\eeqa
The first equality follows from Tonelli's theorem and \(\int_{\mathcal Y}g_t(\mathrm y_t\mid x)\,\mathrm d \mathrm y_t=1\), while the second uses the total variation convention of Section~\ref{ssNotation} and the densities \(\mu(g_{\mathrm y_t})\) and \(\nu(g_{\mathrm y_t})\) of \(L_t(\mu)\) and \(L_t(\nu)\), respectively.
\end{proof}

This result is very similar to Lemma~3.5 in McDonald and Y\"uksel (2020)~\cite{mcdonald2020exponential}; we include an explicit proof because we do not assume \(\mu\ll\nu\).

Although Lemma~\ref{lem:Emu} is stated for deterministic
\(\mu,\nu\in\mP(\mX)\), it applies pathwise to random probability measures of
the form \(\mu_{Y_{1:t-1}}\) and \(\nu_{Y_{1:t-1}}\) (indeed, the proof
relies on the pointwise estimate of Proposition~\ref{DUI}). Hence, after
fixing a realization \(Y_{1:t-1}=y_{1:t-1}\), the bound holds for the
corresponding measures \(\mu_{y_{1:t-1}}\) and \(\nu_{y_{1:t-1}}\), and may
therefore be used conditionally in the filtering recursion.

Lemma~\ref{lem:Emu} provides the standard route to stability estimates; see,
e.g., \cite{mcdonald2020exponential}. In particular, when combined with the
contraction of the predictive operator, it yields a one-step estimate with
factor
\(
(1+\|L_t\|)\|K_t\|.
\)
Consequently, non-expansivity of the prediction--update composition is
guaranteed whenever
\(
(1+\|L_t\|)\|K_t\|\leq1.
\)
For instance, since \(\|L_t\|\leq1\), the sufficient condition
\(
\|K_t\|\leq\frac12
\)
ensures non-expansivity. Such a condition, however, requires a rather strong
contraction of the transition kernel.

This motivates the refined analysis below. The following theorem establishes an identity for the conditional expectation of the total variation distance between two update operators evaluated at general random probability measures. It enables the identification of conditions under which the update step is itself non-expansive, thereby reducing the contraction required from the transition kernels.

Let \(\mu,\nu\in\mathcal P(\mX)\) be \(\mathcal G_{t-1}\)-measurable random probability measures. Let \((\widetilde A_t,\widetilde B_t)\) and \((A_t,B_t)\) be measurable Hahn decompositions of \(\mu-\nu\) and \(U_{Y_t}(\mu)-U_{Y_t}(\nu)\) respectively, with positive sets \(\widetilde A_t\) and \(A_t\) chosen such that their indicators are jointly measurable in all their arguments. We define the \emph{remainder} and \emph{covariance} terms respectively as
\beq
\left.
	\begin{array}{lll}
	\mathcal{T}_{{\rm rem},t}(\mu, \nu) 
	&:= & \mathbb E^{L_t(\mu)} \left[
		\frac{
			|\mu-\nu|
			\left(
				g_{Y_t}\mathbf 1_{A_t\triangle\widetilde A_t}
			\right)
		}{
			\mu(g_{Y_t})
		}
	\right] \quad \text{and} \\
	\mathcal{T}_{{\rm cov},t}(\mu, \nu) 
	&:= &\operatorname{Cov}^{L_t(\mu)}\left(
		\frac{
			\nu(g_{Y_t})
		}{
			\mu(g_{Y_t})
		}, U_{Y_t}(\nu)(A_t)
	\right),
	\end{array}
\right.
\label{eq:Trem_Tcov}
\eeq
and let the \emph{update discrepancy functional} be the difference 
\begin{equation}
\label{eq:update-discrepancy}
\mathcal{T}_t(\mu, \nu) := \mathcal{T}_{{\rm rem},t}(\mu, \nu) - \mathcal{T}_{{\rm cov},t}(\mu, \nu).
\end{equation}

The following Theorem shows how this discrepancy directly controls the expansivity of the update operator.

\begin{theorem}[Bayesian TV balance identity]\label{lemma:ExUpd} With the preceding notation,
\begin{equation}
\label{eq:conditional-update-identity}
\mathbb E^{L_t(\mu)}
\left[
\left\|
U_{Y_t}(\mu)
-
U_{Y_t}(\nu)
\right\|_{\mathrm{TV}}
\right]
=
\|\mu-\nu\|_{\mathrm{TV}}
- 
\mathcal{T}_t(\mu, \nu) \quad \mathbb P\text{-a.s.},
\end{equation}
and, as a consequence, 
\[
\mathbb E^{L_t(\mu)}
\left[
\|U_{Y_t}(\mu)-U_{Y_t}(\nu)\|_{\mathrm{TV}}
\right]
\leq
\|\mu-\nu\|_{\mathrm{TV}}
\quad \text{if, and only if,} \quad 
\mathcal{T}_t(\mu, \nu) \geq 0 \quad \mathbb P\text{-a.s.}
\]
\end{theorem}
A proof is provided in Appendix~\ref{proof:ExUTV}.

The characterization in Theorem~\ref{lemma:ExUpd} admits a useful interpretation. The non-expansivity of the update operator depends entirely on the sign of the discrepancy functional \(\mathcal{T}_t(\mu, \nu)\), which expresses a balance between two effects:
\begin{itemize} 
\item The covariance term \(\mathcal{T}_{{\rm cov},t}(\mu, \nu)\) measures the interaction between the relative predictive likelihood \(\nu(g_{Y_t})/\mu(g_{Y_t})\) and the posterior mass \(U_{Y_t}(\nu)(A_t)\) assigned by \(\nu\) to the positive Hahn set of \(U_{Y_t}(\mu)-U_{Y_t}(\nu)\). 
\item The remainder term \(\mathcal{T}_{{\rm rem},t}(\mu, \nu)\) quantifies the weighted mass of \(|\mu-\nu|\) carried by \(A_t\triangle\widetilde A_t\), namely, the part of the state space whose Hahn classification changes after the update.
\end{itemize}
Thus, non-expansivity of \(U_{Y_t}\) holds whenever \(\mathcal{T}_t(\mu, \nu) \geq 0\). In particular, a non-positive covariance is automatically favourable, since the remainder term (the mass correction) is always non-negative.

We now specialize the preceding general characterization to the filtering recursion generated by the true model. Hereafter, \(\pi_0\) denotes the true initial distribution and
\[
\xi_{Y_{1:t-1}}
:=
K_t\bigl(\Phi_{Y_{1:t-1}}(\pi_0)\bigr)
\]
denotes the corresponding true predictive random probability measure of the state $X_t$. Since the observations are generated by the true model, \(L_t(\xi_{Y_{1:t-1}})\) is precisely the conditional law of \(Y_t\) given \(\mathcal G_{t-1}\). Consequently, when \(\mu=\xi_{Y_{1:t-1}}\), the \(L_t(\mu)\)-expectation appearing in Theorem~\ref{lemma:ExUpd} coincides with the conditional expectation under the true model, given the observation history. This motivates the following classes of random probability measures.

\begin{definition}
\label{def:expected-nonexpansive-class}
\label{def:expected-nonexpansive-class-weak}
Let \(\xi_{Y_{1:t-1}} = K_t\bigl(\Phi_{Y_{1:t-1}}(\pi_0)\bigr)\) denote the true predictive random probability measure at time \(t\geq1\). We define \(\mathfrak U(\xi_{Y_{1:t-1}})\) as the set of all \(\mathcal G_{t-1}\)-measurable random probability measures \(\nu_{Y_{1:t-1}}\in\mathcal P(\mX)\) satisfying, \(\mathbb P\)-a.s.,
\begin{equation}
\label{eq:nonexpansive-class}
\mathcal{T}_t(\xi_{Y_{1:t-1}}, \nu_{Y_{1:t-1}}) \geq 0.
\end{equation}
We further denote by \(\mathfrak U_{\mathrm E}(\xi_{Y_{1:t-1}})\) the set of all \(\mathcal G_{t-1}\)-measurable random probability measures \(\nu_{Y_{1:t-1}}\in\mathcal P(\mX)\) for which \(\mathbb E[\mathcal{T}_t(\xi_{Y_{1:t-1}}, \nu_{Y_{1:t-1}})] \geq 0\). Clearly, \(\mathfrak U(\xi_{Y_{1:t-1}}) \subseteq \mathfrak U_{\mathrm E}(\xi_{Y_{1:t-1}})\).
\end{definition}

Note that $\mT_{{\rm cov},t}(\xi_{Y_{1:t-1}}, \nu_{Y_{1:t-1}}) \le 0$ immediately implies $\mT_t(\xi_{Y_{1:t-1}}, \nu_{Y_{1:t-1}}) \ge 0$ and, hence, $\nu_{Y_{1:t-1}} \in \mathfrak U(\xi_{Y_{1:t-1}})$. Membership in these classes not only guarantees non-expansivity, but their exact characterizations allow us to quantify the precise rates of contraction induced by the update operator, both conditionally and unconditionally.

\begin{definition}[Contraction rates]
\label{def:contraction-rates}
For any $\nu_{Y_{1:t-1}}\in\mathfrak U(\xi_{Y_{1:t-1}})$,
we define the \emph{conditional contraction rate}
$\delta(\nu_{Y_{1:t-1}})$ as the random variable
\begin{equation}
\label{eq:delta-definition}
\delta(\nu_{Y_{1:t-1}}):=
\begin{cases}
\displaystyle
\frac{
\mathcal T_t(\xi_{Y_{1:t-1}},\nu_{Y_{1:t-1}})
}{
\|\xi_{Y_{1:t-1}}-\nu_{Y_{1:t-1}}\|_{\mathrm{TV}}
},
&\text{if}~~
\|\xi_{Y_{1:t-1}}-\nu_{Y_{1:t-1}}\|_{\mathrm{TV}}>0,\\[2ex]
0, &\text{otherwise}.
\end{cases}
\end{equation}
Similarly, for any $\nu_{Y_{1:t-1}}\in\mathfrak U_{\mathrm E} (\xi_{Y_{1:t-1}})$, we define the \emph{unconditional contraction rate} $\bar\delta(\nu_{Y_{1:t-1}})$ as the deterministic quantity 
\beq
\label{eq:bar-delta-definition}
\bar\delta(\nu_{Y_{1:t-1}}):=
\begin{cases}
\displaystyle
\frac{
\mbE\left[
	\mathcal T_t(\xi_{Y_{1:t-1}},\nu_{Y_{1:t-1}})
\right]
}{
\mbE\left[
	\|\xi_{Y_{1:t-1}}-\nu_{Y_{1:t-1}}\|_{\mathrm{TV}}
\right]
},
&\text{if}~~
\mbE\left[
	\|\xi_{Y_{1:t-1}}-\nu_{Y_{1:t-1}}\|_{\mathrm{TV}}
\right] >0,\\[2ex]
0, &\text{otherwise}.
\end{cases}
\eeq
\end{definition}

By Theorem~\ref{lemma:ExUpd} and the admissibility conditions defining $\mathfrak U(\xi_{Y_{1:t-1}})$ and $\mathfrak U_{\mathrm E}(\xi_{Y_{1:t-1}})$, both contraction rates are well-defined and take values in $[0,1]$.

\begin{lemma}\label{RIE_G}
Let $\xi_{Y_{1:t-1}}=K_t(\Phi_{Y_{1:t-1}}(\pi_0))$ denote the true predictive random probability measure, and let $\mathcal G_{t-1}=\sigma(Y_{1:t-1})$. If $\nu_{Y_{1:t-1}}\in\mathfrak U(\xi_{Y_{1:t-1}})$, then, $\mathbb P$-a.s.,
\begin{equation}
\label{eq1*}
\begin{aligned}
&\mathbb E\!\left[
\left.
\|U_{Y_t}(\xi_{Y_{1:t-1}})
-U_{Y_t}(\nu_{Y_{1:t-1}})\|_{\mathrm{TV}}
\,\right|\,\mathcal G_{t-1}
\right] = 
(1-\delta(\nu_{Y_{1:t-1}}))
\|\xi_{Y_{1:t-1}}-\nu_{Y_{1:t-1}}\|_{\mathrm{TV}}.
\end{aligned}
\end{equation}
Similarly, if
$\nu_{Y_{1:t-1}}\in
\mathfrak U_{\mathrm E}(\xi_{Y_{1:t-1}})$,
then
\begin{equation}
\label{eq2ndclaim}
\begin{aligned}
&\mathbb E\!\left[
\|U_{Y_t}(\xi_{Y_{1:t-1}})
-U_{Y_t}(\nu_{Y_{1:t-1}})\|_{\mathrm{TV}}
\right]
=
(1-\bar\delta(\nu_{Y_{1:t-1}}))
\mathbb E\!\left[
\|\xi_{Y_{1:t-1}}-\nu_{Y_{1:t-1}}\|_{\mathrm{TV}}
\right].
\end{aligned}
\end{equation}
\end{lemma}

\begin{proof}
For brevity, write
$\xi=\xi_{Y_{1:t-1}}$ and
$\nu=\nu_{Y_{1:t-1}}$.
By Proposition~\ref{prop:Ypdf},
$L_t(\xi)$ is the conditional law of $Y_t$
given $\mathcal G_{t-1}$. Hence,
Theorem~\ref{lemma:ExUpd} yields
\begin{equation}
\label{eqlong}
\mathbb E\!\left[
\left.
\|U_{Y_t}(\xi)-U_{Y_t}(\nu)\|_{\mathrm{TV}}
\,\right|\,\mathcal G_{t-1}
\right]
=
\|\xi-\nu\|_{\mathrm{TV}}
-\mathcal T_t(\xi,\nu).
\end{equation}
The identity \eqref{eq1*} then follows directly from the definition of $\delta(\nu)$, while taking expectations in \eqref{eqlong} and applying the definition of $\bar\delta(\nu)$ yields \eqref{eq2ndclaim}.
\end{proof}


We now combine Lemmas \ref{lem:Emu} and \ref{RIE_G}. Let \(\pi_0\in\mathcal P(\mathcal X)\) denote the true initial distribution and let \(\widehat{\pi}_0\in\mathcal P(\mathcal X)\) be an arbitrary alternative initial distribution. For every \(i\geq1\), define
\[
\xi_{Y_{1:i-1}}
:=
K_i(\Phi_{Y_{1:i-1}}(\pi_0))
\quad \text{and} \quad
\widehat\xi_{Y_{1:i-1}}
:=
K_i(\Phi_{Y_{1:i-1}}(\widehat{\pi}_0)).
\]
Thus, \(\xi_{Y_{1:i-1}}\) is the true predictive random probability measure, whereas \(\widehat\xi_{Y_{1:i-1}}\) is the predictive measure obtained from the same filtering recursion with the alternative initialization \(\widehat{\pi}_0\). For \(t\geq1\), let
\beq
\mathcal I_t
:=
\left\{
i\in\{1,\ldots,t\}:
\widehat\xi_{Y_{1:i-1}}
\notin
\mathfrak U_{\mathrm E}(\xi_{Y_{1:i-1}})
\right\}
\label{defIt}
\eeq
denote the set of update times up to time \(t\) at which the update is not non-expansive in expectation, and let
$
\mI_t^c = \{1, \ldots, t\} \backslash \mI_t
$
be its complement, so that $\mI_t \cup \mI_t^c = \{ 1, \ldots, t \}$ and $\mI_t \cap \mI_t^c = \emptyset$. For \(i \in \mathcal I_t^c\), let \(\bar{\delta}_i := \bar{\delta}(\widehat\xi_{Y_{1:i-1}})\) denote the corresponding unconditional contraction rate defined in \eqref{def:contraction-rates}. Finally, denote
\(
a_0:=\|\pi_0-\widehat{\pi}_0\|_{\mathrm{TV}}.
\)

\begin{theorem}\label{SRF_TB}
For a given initialization $\widehat \pi_0$ and every \(t\geq1\),
\beq
\mathbb E\!\left[
\left\|
\Phi_{Y_{1:t}}(\pi_0)
-
\Phi_{Y_{1:t}}(\widehat{\pi}_0)
\right\|_{\mathrm{TV}}
\right]
\leq
a_0
\prod_{i=1}^{t}\|K_i\|
\prod_{i\in\mathcal I_t^c}
\bigl(1-\bar{\delta}_i\bigr)
\prod_{i\in\mathcal I_t}
\bigl(1+\|L_i\|\bigr)
\label{eq1sost}
\eeq
and, as a consequence,
\beqa
\mathbb E\left[
\left\|
	\Phi_{Y_{1:t}}(\pi_0)
	-
	\Phi_{Y_{1:t}}(\widehat{\pi}_0)
	\right\|_{\mathrm{TV}}
\right]
&\leq&
a_0 \exp\left\{
	-\sum_{i=1}^{t}\bigl(1-\|K_i\|\bigr)
	-\sum_{i\in\mathcal I_t^c}\bar{\delta}_i
\right.
\nn
\\
&& \qquad \qquad \qquad \qquad \left.
	+ \sum_{i\in\mathcal I_t} \log\bigl(1+\|L_i\|\bigr)
\right\}.
\label{eq2sost}
\eeqa
In particular, if
\beq
\lim_{t\to\infty}\left[
\sum_{i=1}^{t}\bigl(1-\|K_i\|\bigr)
+
\sum_{i\in\mathcal I_t^c}\bar{\delta}_i
-
\sum_{i\in\mathcal I_t}
\log\bigl(1+\|L_i\|\bigr)
\right]=\infty,
\label{eq2-*}
\eeq
then
\[
\lim_{t\to\infty} \mathbb E\!\left[
\left\|
\Phi_{Y_{1:t}}(\pi_0)
-
\Phi_{Y_{1:t}}(\widehat{\pi}_0)
\right\|_{\mathrm{TV}}
\right]
= 0.
\]
\end{theorem}

\begin{proof}
Choose some $t \ge 1$ and, for \(1 \le i \le t\), let $a_i := \mathbb{E}\left[
\left\|
\Phi_{Y_{1:i}}(\pi_0)
-
\Phi_{Y_{1:i}}(\widehat{\pi}_0)
\right\|_{\mathrm{TV}}
\right]$. If \(i\in\mathcal I_t^c\), then
\(
\widehat\xi_{Y_{1:i-1}}
\in
\mathfrak U_{\mathrm E}(\xi_{Y_{1:i-1}})
\),
and Lemma~\ref{RIE_G}, together with \eqref{ineq:Psi_bound}, gives
\[
a_i
\leq
(1-\bar{\delta}_i)
\mathbb E\!\left[
\|\xi_{Y_{1:i-1}}-\widehat\xi_{Y_{1:i-1}}\|_{\mathrm{TV}}
\right]
\leq
(1-\bar{\delta}_i)\|K_i\|a_{i-1}.
\]

If \(i\in\mathcal I_t\), the non-expansive estimate is not available.
Applying Lemma~\ref{lem:Emu} conditionally w.r.t. the true predictive measure, and then using the tower property and \eqref{ineq:Psi_bound}, yields
\[
a_i
\leq
(1+\|L_i\|)
\mathbb E\!\left[
\|\xi_{Y_{1:i-1}}-\widehat\xi_{Y_{1:i-1}}\|_{\mathrm{TV}}
\right]
\leq
(1+\|L_i\|)\|K_i\|a_{i-1}.
\]
Iterating these two estimates for \(i=1,\ldots,t\) gives
\[
a_t
\leq
a_0
\prod_{i=1}^{t}\|K_i\|
\prod_{i\in\mathcal I_t^c}
\bigl(1-\bar{\delta}_i\bigr)
\prod_{i\in\mathcal I_t}
\bigl(1+\|L_i\|\bigr),
\]
which proves \eqref{eq1sost}. 

Since \(x\leq e^{-(1-x)}\) for every \(x\in[0,1]\) and \(1+x \leq e^x\) for \(x\ge 0\), we obtain
\[
\prod_{i=1}^{t}\|K_i\|
\leq
\exp\left(
-\sum_{i=1}^{t}(1-\|K_i\|)
\right),
\quad 
\prod_{i\in\mathcal I_t^c}(1-\bar{\delta}_i)
\leq
\exp\left(
-\sum_{i\in\mathcal I_t^c}\bar{\delta}_i
\right),
\]
and 
\[
\prod_{i\in\mathcal I_t}(1+\|L_i\|)
=
\exp\left(
\sum_{i\in\mathcal I_t}\log(1+\|L_i\|)
\right).
\]
Therefore,
\[
a_t
\leq
a_0
\exp\left\{
-\sum_{i=1}^{t}(1-\|K_i\|)
-
\sum_{i\in\mathcal I_t^c}\bar{\delta}_i
+
\sum_{i\in\mathcal I_t}\log(1+\|L_i\|)
\right\},
\]
which proves \eqref{eq2sost}. Moreover, condition \eqref{eq2-*} implies that $a_t \to 0$. 
\end{proof}

\begin{remark}
From the proof of Theorem \ref{SRF_TB}, it is apparent that if
\(
\widehat\xi_{Y_{1:i-1}}
\in
\mathfrak U_{\mathrm E}(\xi_{Y_{1:i-1}})
\)
for every \(1 \le i \le t\), then \(\mathcal I_t=\emptyset\) for every \(t\),
and the filtering error estimate reduces to
\[
\begin{aligned}
\mathbb E\!\left[
\left\|
\Phi_{Y_{1:t}}(\pi_0)
-
\Phi_{Y_{1:t}}(\widehat{\pi}_0)
\right\|_{\mathrm{TV}}
\right]
&\leq
a_0
\prod_{i=1}^{t}\bigl[(1-\bar{\delta}_i)\|K_i\|\bigr]
\\
&\leq
a_0
\exp\left(
-\sum_{i=1}^{t}\bigl(1-\|K_i\| + \bar{\delta}_i\bigr)
\right).
\end{aligned}
\]
Hence, in this case, 
\(
\sum_{i=1}^{\infty}(1-\|K_i\| + \bar{\delta}_i)=\infty
\)
is sufficient for convergence of the expected filtering error. Notably, this implies that coalescence of $\Phi_{Y_{1:t}}(\pi_0)$ and $\Phi_{Y_{1:t}}(\widehat \pi_0)$ can be achieved even if the transition kernels are completely non-contractive i.e., $\|K_i\| = 1$, (as can be the case, for instance, for some Gaussian kernels), provided the update step supplies sufficient contraction via $\sum_{i=1}^\infty \bar{\delta}_i = \infty$. 
\end{remark}

Theorem~\ref{SRF_TB} shows that, relative to the true filtering trajectory, the convergence mechanism is governed by two competing effects: the combined contraction produced by the transition kernels and the updates against the potential expansion when $\widehat\xi_{Y_{1:t-1}} \notin \mathfrak{U}_E(\xi_{Y_{1:t-1}})$. Convergence is guaranteed when the accumulated contraction dominates the accumulated potential expansion (i.e., when \eqref{eq2-*} holds). In particular, it is not necessary for every update to be non-expansive in expectation.

Nevertheless, Theorem~\ref{SRF_TB} only provides a convergence result relative to a fixed alternative initialization $\widehat \pi_0$, since the set \(\mathcal I_t\) and the rates \(\bar{\delta}_i\) depend precisely on the law $\widehat \pi_0$. Its convergence criterion yields stability in the sense of Definition \ref{def:Ex_St} only when verified for every alternative initialization. 

Obtaining explicit and general conditions on the model parameters that guarantee membership in \(\mathfrak U(\xi_{Y_{1:t-1}})\), or in the weaker expected class \(\mathfrak U_{\mathrm E}(\xi_{Y_{1:t-1}})\), is in general nontrivial and requires further analysis. These conditions can be verified in relevant settings, though, and in Section \ref{sec5:Linear_Gaussian} we provide some examples. 

%
\subsection{Filtering error process}\label{sec3:Error_process}

Given the true initial distribution
\(\pi_0\in\mathcal P(\mathcal X)\) and an arbitrary alternative
initialization
\(\widehat\pi_0\in\mathcal P(\mathcal X)\), denote the corresponding
filtering measures as
\[
\pi_{Y_{1:t}}
:=
\Phi_{Y_{1:t}}(\pi_0)
\quad \text{and} \quad 
\widehat\pi_{Y_{1:t}}
:=
\Phi_{Y_{1:t}}(\widehat\pi_0),
\]
and introduce the filtering error process \(\{E_t\}_{t\geq0}\), where
\begin{equation}
\label{eq:filtering-error-process}
E_0 := \| \pi_0 - \widehat \pi_0 \|_{\mathrm{TV}}
\quad \text{and} \quad
E_t
:=
\left\|
\Phi_{Y_{1:t}}(\pi_0)
-
\Phi_{Y_{1:t}}(\widehat\pi_0)
\right\|_{\mathrm{TV}},
\quad \text{for} \quad t\ge 1.
\end{equation}
This process is adapted to the observation filtration
\(
\mG:=\{\mathcal G_t\}_{t\geq0},
\)
where \(\mathcal G_0\) denotes the trivial \(\sigma\)-algebra. Since \(E_t\)
is the total variation distance between two probability measures,
\(
0\leq E_t\leq1
\)
\(\mathbb P\)-a.s. for every \(t\geq0\).

We next establish several convergence properties for uniformly bounded non-negative r.v.s. In particular, these results apply to the filtering error process \(\{E_t\}_{t\geq0}\).

\begin{lemma}\label{lem:Bounded_Conv_Equiv}
Let \(\{Z_t\}_{t\geq0}\) be a sequence of non-negative r.v.s such that
\[
0\leq Z_t \le 1
\quad
\text{a.s. for every $t\ge 0$}.
\]
Then the following statements are equivalent:
\begin{enumerate}[label=\roman*)]
    \item
    \(
    Z_t\xrightarrow{\mathbb P}0;
    \)

    \item
    \(\exists p_0\geq1\) such that
    \(
    Z_t\xrightarrow{L^{p_0}}0;
    \)

    \item
    \(\forall p\geq1\),
    \(
    Z_t\xrightarrow{L^p}0;
    \)

    \item
    \(\forall p\geq1\),
    \(
    Z_t^p\xrightarrow{\mathbb P}0.
    \)
\end{enumerate}
Under any of these equivalent conditions,
\(
\operatorname{Var}(Z_t)\rightarrow 0.
\)
Moreover,
$$
\text{if} \quad
\sum_{t=0}^{\infty}
\mathbb P(Z_t\geq\delta)
<\infty ~~\forall\delta>0,
\quad \text{then} \quad
Z_t^p\xrightarrow{\mathrm{a.s.}}0 ~~\forall p\geq1.
$$
\end{lemma}
A proof is provided in Appendix~\ref{proof:Bounded_Conv_Equiv}.

Since the filtering error process satisfies \(E_t\in[0,1]\) a.s., the previous lemma can be used to obtain a characterization of stability.

\begin{theorem}\label{thm:Stability_ConvProb}
The sequence \(\{\Phi_{Y_t}\}_{t\geq1}\) is stable in the sense of
Definition~\ref{def:Ex_St} if and only if, for every alternative initial
distribution
\(
\widehat\pi_0\in\mathcal P(\mathcal X),
\)
the corresponding filtering error process defined in
\eqref{eq:filtering-error-process} satisfies
\(
E_t\xrightarrow{\mathbb P}0.
\)
Moreover, under these equivalent conditions,
\(
\operatorname{Var}(E_t)\to0
\)
for every alternative initialization \(\widehat\pi_0\).
\end{theorem}

\begin{proof}
Suppose first that the random filtering operators are stable in the sense
of Definition~\ref{def:Ex_St}. Taking
\(\mu_0=\pi_0\) and
\(\nu_0=\widehat\pi_0\) in that definition gives
\(
E_t\xrightarrow{L^1}0
\)
for every
\(\widehat\pi_0\in\mathcal P(\mathcal X)\). Hence,
Lemma~\ref{lem:Bounded_Conv_Equiv} implies
\(
E_t\xrightarrow{\mathbb P}0.
\)

Conversely, suppose that
\(
E_t\xrightarrow{\mathbb P}0
\)
for every alternative initialization
\(\widehat\pi_0\in\mathcal P(\mathcal X)\). Since \(0\leq E_t\leq1\),
Lemma~\ref{lem:Bounded_Conv_Equiv} yields
\(
E_t\xrightarrow{L^1}0
\)
for every such initialization. 

Now let \(\mu_0,\nu_0\in\mathcal P(\mathcal X)\) be arbitrary. A simple triangle inequality yields
\beqa
\mbE\left[
	\| \Phi_{Y_{1:t}}(\mu_0) - \Phi_{Y_{1:t}}(\nu_0) \|_{\mathrm{TV}}
\right] &\le& \mbE\left[
	\| \Phi_{Y_{1:t}}(\mu_0) - \Phi_{Y_{1:t}}(\pi_0) \|_{\mathrm{TV}}
\right] \nn \\
&& + \mbE\left[
	\| \Phi_{Y_{1:t}}(\pi_0) - \Phi_{Y_{1:t}}(\nu_0) \|_{\mathrm{TV}}
\right]
\label{eq-triangle-init}
\eeqa
and taking \(\mu_0\) and \(\nu_0\) in turn as alternative initializations relative to the true distribution \(\pi_0\) shows that both terms on the right-hand side of \eqref{eq-triangle-init} converge to zero as $t \to \infty$. Therefore, the operators are stable in the sense of Definition~\ref{def:Ex_St}. 

The convergence of \(\operatorname{Var}(E_t)\) follows again from Lemma~\ref{lem:Bounded_Conv_Equiv}.
\end{proof}

Under the stronger condition
\(
\widehat\xi_{Y_{1:t-1}}
\in
\mathfrak U(\xi_{Y_{1:t-1}})
\)
for all \(t\geq1\), the filtering error process admits an additional probabilistic interpretation: it is a supermartingale w.r.t. the observation filtration.

\begin{remark}\label{rem:super_Martingale}
Assume that
\(
\widehat\xi_{Y_{1:t-1}}
\in
\mathfrak U(\xi_{Y_{1:t-1}})
\)
for all \(t\geq1\). Then, by the conditional estimate in
Lemma~\ref{RIE_G}, together with \eqref{ineq:Psi_bound} and
\(\|K_t\|\leq1\),
\[
\mathbb E\!\left[
E_t\,\middle|\,\mathcal G_{t-1}
\right]
\leq
\|K_t\|E_{t-1}
\leq
E_{t-1},
\qquad t\geq1.
\]
Hence, \(\{E_t\}_{t\geq0}\) is a non-negative
\(\mG\)-supermartingale according to the definition in \cite{Williams91}.
\end{remark}

%
\section{Model misspecification}\label{sec4}

In this section, we formalize the notion of model misspecification within the
random-observation framework of Section~\ref{sec2}. We compare the optimal
filtering recursion generated by the true model with a perturbed recursion
arising from misspecified dynamics and a possibly incorrect initial
distribution. Both recursions are evaluated along the observation process
generated by the true model. Building on the stability results of
Section~\ref{sec3}, we derive quantitative bounds on the resulting filtering
error and show how the contraction induced by the predictive operators
competes with the possible expansion of the update operators and the
perturbations introduced by model misspecification.

%
\subsection{True vs misspecified model}

In addition to the true model
\(
\mathcal M=(\pi_0,K,g)
\)
introduced in Section~\ref{sec2}, we consider a misspecified SSM on the same
probability space
\(
(\Omega,\mathcal F,\mathbb P).
\)
Let
\(
\widehat X=\{\widehat X_t\}_{t\geq0}
\)
be a second state process taking values in \(\mathcal X\), with initial
distribution \(\widehat\pi_0\) and Markov kernels
\(
\widehat K=\{\widehat K_t\}_{t\geq1}.
\)
The state process is defined by
\beq
    \widehat X_0 \sim \widehat \pi_0(\mathrm d x_0)
    \quad \text{and} \quad
    \widehat X_t\mid\widehat X_{t-1} \sim \widehat K_t(\widehat X_{t-1},\mathrm d x_t),
\nn
\eeq
while the observations
\(
\{Y_t\}_{t\geq1}
\)
are generated by the true model \(\mM\). We assume that the likelihoods are
correctly specified, so that the misspecified model is
$
\widehat{\mathcal M}
:=
(\widehat\pi_0,\widehat K,g)
$.
Throughout this section, all conditional and unconditional expectations are
taken w.r.t. the probability law of the random observations $\{Y_t\}_{t \ge 1}$ generated by the true model \(\mathcal M\).

For the model \(\widehat\mM\), we define the prediction, update and
prediction--update operators, in analogy with \eqref{PO}, \eqref{EUO} and
\eqref{PEUO}, respectively, as
\[
\widehat K_t(\mu),
\quad
U_{Y_t}(\mu)
:=
\frac{g_{Y_t}\cdot\mu}{\mu(g_{Y_t})}
\quad \text{and} \quad
\widehat\Phi_{Y_t}(\mu)
:=
U_{Y_t}\circ\widehat K_t(\mu).
\] 
In particular, the update operator is the same for the true and misspecified
models because the likelihoods \(g=\{g_t\}_{t\ge 1}\) are assumed to be correctly specified.
We further define the iterated random prediction--update misspecified
operator by
\begin{equation}\label{RPEUep}
\widehat\Phi_{Y_{k:t}}
:=
\widehat\Phi_{Y_t}
\circ
\widehat\Phi_{Y_{t-1}}
\circ\cdots\circ
\widehat\Phi_{Y_k},
\qquad
1\leq k\leq t.
\end{equation}

\begin{remark}\label{rem:miss_K_contractive}
The predictive operators
\(
\widehat K_t:\mathcal D(\mathcal X)\to\mathcal D(\mathcal X)
\)
are linear and non-expansive in total variation. There exists a constant
\(
0\leq C_{\widehat K}\leq1
\)
such that
\(
\|\widehat K_t\|\leq C_{\widehat K}
\)
for every \(t\geq1\).
\end{remark}

In the remainder of this section, we compare the optimal filtering sequence
\(
\{\pi_{Y_{1:t}}\}_{t\geq1}
\)
generated by the true model with the misspecified sequence
\(
\{\widehat\pi_{Y_{1:t}}\}_{t\geq1},
\)
where
\[
\widehat\pi_{Y_{1:t}}
=
\widehat\Phi_{Y_t}(\widehat\pi_{Y_{1:t-1}})
=
U_{Y_t}\circ
\widehat K_t(\widehat\pi_{Y_{1:t-1}})
=
\widehat\Phi_{Y_{1:t}}(\widehat\pi_0),
\qquad t\geq1.
\]
While the two filtering recursions may differ in both their initial distributions and predictive dynamics, they are evaluated along the same observation sequence generated by the true model. 

%
\subsection{Perturbation operator}\label{ssce:Dif_P}

Define the linear operator
\(
\Delta K_t:\mbM(\mX)\to\mbM(\mX)
\)
as
\[
\Delta K_t:=K_t-\widehat K_t,
\]
which represents the discrepancy between the true and misspecified
predictive dynamics at time \(t\). Although \(\Delta K_t\) is defined on
\(\mbM(\mX)\), only its action on probability measures is used below. Since
\(K_t(\nu)\) and \(\widehat K_t(\nu)\) are probability measures for every
\(\nu\in\mP(\mX)\),
\(
\|\Delta K_t(\nu)\|_{\mathrm{TV}}\leq1.
\)
Accordingly, for \(t\geq1\) define
\begin{equation}\label{rem:unibound-Delta2}
\varepsilon_{t-1}
:=
\sup_{\nu\in\mP(\mX)}
\|\Delta K_t(\nu)\|_{\mathrm{TV}}
\quad \text{and} \quad
\varepsilon
:=
\sup_{t\geq0}\varepsilon_t.
\end{equation}
Then
\(
\|\Delta K_t(\nu)\|_{\mathrm{TV}}
\leq\varepsilon_{t-1}\leq\varepsilon
\)
for every \(\nu\in\mP(\mX)\) and \(t\geq1\).

The bound \(\varepsilon_{t-1}\) is uniform over all probability measures,
whereas the analysis below also uses the discrepancy evaluated along the
true optimal filtering trajectory. Let
\(
\{\pi_{Y_{1:t}}\}_{t\geq1}
\)
denote the random optimal filtering sequence and, for \(t\geq1\), define
\begin{equation}\label{rem:unibound-Delta}
\widehat\varepsilon_{t-1}
:=
\|\Delta K_t(\pi_{Y_{1:t-1}})\|_{\mathrm{TV}}
\quad \text{and} \quad
\widehat\varepsilon
:=
\sup_{t\geq0}\widehat\varepsilon_t.
\end{equation}
Since
\(
\pi_{Y_{1:t-1}}\in\mP(\mX),
\)
\eqref{rem:unibound-Delta2} implies
\(
0\leq\widehat\varepsilon_t\leq\varepsilon_t
\)
for every \(t\geq0\), and therefore
\(
0\leq\widehat\varepsilon\leq\varepsilon\leq1.
\)
Here \(\widehat\varepsilon_{t-1}\) measures the discrepancy between the true
and misspecified predictive operators ($K_t$ and $\widehat K_t$, respectively) when both are evaluated at the true filtering measure \(\pi_{Y_{1:t-1}}\), whereas \(\widehat\varepsilon\) is its uniform-in-time bound.

The uniform bounds \(\varepsilon\) and \(\widehat\varepsilon\) are used in
Section~\ref{sec:Stability_Miss}, while the time-dependent sequences
\(\{\varepsilon_t\}_{t\geq0}\) and
\(\{\widehat\varepsilon_t\}_{t\geq0}\) are considered in
Section~\ref{sscec:Decay_missp}.

\begin{proposition}\label{PPofE}
Let \(\nu\in\mP(\mX)\). Then
\begin{align*}
\|K_t(\pi_{Y_{1:t-1}})
-
\widehat K_t(\nu)\|_{\mathrm{TV}}
&\leq
\|\widehat K_t\|\,
\|\pi_{Y_{1:t-1}}-\nu\|_{\mathrm{TV}}
+
\widehat\varepsilon_{t-1}
\\
&\leq
C_{\widehat K}\,
\|\pi_{Y_{1:t-1}}-\nu\|_{\mathrm{TV}}
+
\widehat\varepsilon,
\end{align*}
where \(C_{\widehat K}\) is the uniform bound for
\(\|\widehat K_t\|\) introduced in
Remark~\ref{rem:miss_K_contractive}.
\end{proposition}

\begin{proof}
By the triangle inequality,
\beq
\|K_t(\pi_{Y_{1:t-1}})
-
\widehat K_t(\nu)\|_{\mathrm{TV}}
\leq
\|K_t(\pi_{Y_{1:t-1}})
-
\widehat K_t(\pi_{Y_{1:t-1}})\|_{\mathrm{TV}}
+
\|\widehat K_t(\pi_{Y_{1:t-1}})
-
\widehat K_t(\nu)\|_{\mathrm{TV}}.
\nn
\eeq
The first term on the right hand side of the inequality is simply
\(
\|\Delta K_t(\pi_{Y_{1:t-1}})\|_{\mathrm{TV}},
\)
which is bounded by
\(\widehat\varepsilon_{t-1}\leq\widehat\varepsilon\). By
\eqref{ineq:Psi_bound}, the second term on the right hand side can be bounded as
\[
\|\widehat K_t(\pi_{Y_{1:t-1}})
-
\widehat K_t(\nu)\|_{\mathrm{TV}}
\leq
\|\widehat K_t\|\,
\|\pi_{Y_{1:t-1}}-\nu\|_{\mathrm{TV}},
\]
which yields the first inequality in the statement of Proposition \ref{PPofE}. The second inequality in the statement follows immediately from
\(\|\widehat K_t\|\leq C_{\widehat K}\).
\end{proof}

As with the predictive estimates of Section~\ref{sec3}, the preceding bound also applies pathwise when \(\nu\) is a \(\mathcal G_{t-1}\)-measurable random probability measure. Proceeding analogously, we obtain the following complementary estimate.

\begin{proposition}\label{PPofE2}
For any \(\mu,\nu\in\mP(\mX)\),
\begin{align*}
\|K_t(\mu)-\widehat K_t(\nu)\|_{\mathrm{TV}}
&\leq
\|K_t\|\,
\|\mu-\nu\|_{\mathrm{TV}}
+
\varepsilon_{t-1}
\\
&\leq
C_K\,
\|\mu-\nu\|_{\mathrm{TV}}
+
\varepsilon,
\end{align*}
for all \(t\geq1\), where \(C_K\) is the uniform bound for
\(\|K_t\|\) introduced in Remark~\ref{rem:UBT_LP}.
\end{proposition}

%
\subsection{Propagation of misspecification errors}
\label{sec:Stability_Miss}

We now study the discrepancy between the filtering recursion generated by
the true model
\(
\mathcal M=(\pi_0,K,g)
\)
and that generated by the misspecified model
\(
\widehat{\mathcal M}=(\widehat\pi_0,\widehat K,g),
\)
both evaluated along the observation sequence generated by the true model.
The analysis combines the convergence results of Section~\ref{sec3} with the
perturbation estimates for \(\Delta K_t\) derived in
Section~\ref{ssce:Dif_P}. Recall that
\(
\|\mu-\nu\|_{\mathrm{TV}}\leq1
\)
for every \(\mu,\nu\in\mP(\mX)\); hence, an informative bound for the
filtering discrepancy must be strictly smaller than \(1\).

Let us define the misspecified predictive law at time $t$ as $\widehat \xi_{Y_{1:t-1}} := \widehat K_t(\widehat \pi_{Y_{1:t-1}})$ and specifically note that, unlike the alternative predictive laws studied in Section~\ref{sec3}, the measure \(\widehat\xi_{Y_{1:t-1}}\) is now generated by the misspecified predictive operator \(\widehat K_t\). Thus, the discrepancy between \(\xi_{Y_{1:t-1}}\) and \(\widehat\xi_{Y_{1:t-1}}\) reflects both the difference in initialization and the perturbation of the transition dynamics.

For \(t\geq1\), let
\beq
\widehat{\mathcal I}_t
:=
\left\{
i\in\{1,\ldots,t\}:
\widehat\xi_{Y_{1:i-1}}
\notin
\mathfrak U_{\mathrm E}(\xi_{Y_{1:i-1}})
\right\}
\label{eq-4-sost}
\eeq
denote the update times up to \(t\) at which the non-expansive estimate in
expectation is not available, and let
\(\widehat{\mathcal I}_t^c := \{1,\ldots,t\}\backslash\widehat{\mathcal I}_t\)
denote its complement. It is also convenient to write
\beq
\widehat c_i
:=
\begin{cases}
1 - \bar{\delta}_i,
&
\widehat\xi_{Y_{1:i-1}}
\in
\mathfrak U_{\mathrm E}(\xi_{Y_{1:i-1}}),
\\[1mm]
1+\|L_i\|,
&
\widehat\xi_{Y_{1:i-1}}
\notin
\mathfrak U_{\mathrm E}(\xi_{Y_{1:i-1}}),
\end{cases}
\label{eq-3-sost}
\eeq
where \(\bar{\delta}_i := \bar{\delta}(\widehat{\xi}_{Y_{1:i-1}})\) is the unconditional contraction rate defined in Definition \ref{def:contraction-rates}. Thus, \(\widehat c_i=1-\bar{\delta}_i\) at update times for which non-expansivity in expectation holds, whereas
\(\widehat c_i=1+\|L_i\|\) when the general estimate of
Lemma~\ref{lem:Emu} is required.

\begin{lemma}\label{lemma:ExFilter_onestep}
For $\widehat c_t$, $ t \ge 1$, defined as in \eqref{eq-3-sost},
\[
\begin{aligned}
\mbE\!\left[
\|
\Phi_{Y_t}(\pi_{Y_{1:t-1}})
-
\widehat\Phi_{Y_t}(\widehat\pi_{Y_{1:t-1}})
\|_{\mathrm{TV}}
\right]
&\leq
\widehat c_t\,
\mbE\!\left[
\|\xi_{Y_{1:t-1}}
-
\widehat\xi_{Y_{1:t-1}}\|_{\mathrm{TV}}
\right]
\\
&\leq
\widehat c_t
\left(
\|K_t\|\,
\mbE\!\left[
\|\pi_{Y_{1:t-1}}
-
\widehat\pi_{Y_{1:t-1}}\|_{\mathrm{TV}}
\right]
+
\varepsilon_{t-1}
\right).
\end{aligned}
\]
Moreover, if
\(
\widehat\xi_{Y_{1:t-1}}
\in
\mathfrak U(\xi_{Y_{1:t-1}}),
\)
then, \(\mathbb P\)-a.s.,
\[
\begin{aligned}
\mbE\!\left[
\|
\Phi_{Y_t}(\pi_{Y_{1:t-1}})
-
\widehat\Phi_{Y_t}(\widehat\pi_{Y_{1:t-1}})
\|_{\mathrm{TV}}
\,\middle|\,
\mathcal G_{t-1}
\right]
&\leq
(1-\delta_t)
\|\xi_{Y_{1:t-1}}
-
\widehat\xi_{Y_{1:t-1}}\|_{\mathrm{TV}}
\\
&\leq
(1-\delta_t)
\left(
\|K_t\|\,
\|\pi_{Y_{1:t-1}}
-
\widehat\pi_{Y_{1:t-1}}\|_{\mathrm{TV}}
+
 \varepsilon_{t-1}
\right),
\end{aligned}
\]
where \(\delta_t := \delta(\widehat\xi_{Y_{1:t-1}})\) is defined in \eqref{eq:delta-definition}.
\end{lemma}

\begin{proof}
By definition, $\Phi_{Y_t}(\pi_{Y_{1:t-1}})
=
U_{Y_t}(\xi_{Y_{1:t-1}})$ and
$
\widehat\Phi_{Y_t}(\widehat\pi_{Y_{1:t-1}})
=
U_{Y_t}(\widehat\xi_{Y_{1:t-1}})
$. If
\(
\widehat\xi_{Y_{1:t-1}}
\in
\mathfrak U_{\mathrm E}(\xi_{Y_{1:t-1}})
\),
Lemma~\ref{RIE_G} gives the first inequality with
\(\widehat c_t=1-\bar{\delta}_t\). Otherwise, Lemma~\ref{lem:Emu}, applied conditionally
with
\(
\mu=\xi_{Y_{1:t-1}}
\)
and
\(
\nu=\widehat\xi_{Y_{1:t-1}},
\)
and followed by the tower property, gives the same inequality with
\(\widehat c_t=1+\|L_t\|\). The second inequality follows from
Proposition~\ref{PPofE2}, applied pathwise with
\(
\mu=\pi_{Y_{1:t-1}}
\)
and
\(
\nu=\widehat\pi_{Y_{1:t-1}},
\)
and then taking expectations.

If, moreover,
\(
\widehat\xi_{Y_{1:t-1}}
\in
\mathfrak U(\xi_{Y_{1:t-1}})
\),
the conditional part of Lemma~\ref{RIE_G} gives the first conditional
inequality (yielding the factor $1-\delta_t$), while Proposition~\ref{PPofE2}, applied pathwise, yields the
second.
\end{proof}

\begin{remark}\label{rem:a_t}
In the same way as in Section \ref{sec3}, let us denote the mean filtering error, with expectation taken under the observation law generated by the true model, as
\beq
a_t
:=
\mbE\!\left[
\|\pi_{Y_{1:t}}-\widehat\pi_{Y_{1:t}}\|_{\mathrm{TV}}
\right],
\quad \text{with} \quad
a_0:=\|\pi_0-\widehat\pi_0\|_{\mathrm{TV}}.
\label{eq-5-sost}
\eeq
Applying Lemma~\ref{lemma:ExFilter_onestep} yields
\begin{equation}\label{eq:recursion}
a_t
\leq
\widehat c_t
\left(
\|K_t\|a_{t-1}
+
\varepsilon_{t-1}
\right).
\end{equation}
Thus, the filtering error is governed by three effects: the contraction of the predictive operator ($\| K_t \|$) and the update operator ($\widehat c_t$), and the local perturbation in the dynamics ($\varepsilon_{t-1}$). If every update is non-expansive in expectation, then \(\widehat c_t=1-\bar{\delta}_t\) for all \(t\), and
\eqref{eq:recursion} reduces to
\(
a_t\leq (1-\bar{\delta}_t) \left( \|K_t\|a_{t-1}+\varepsilon_{t-1} \right).
\)
If we work with the local perturbation of the dynamics along the sequence of optimal filters, as given by expression \eqref{rem:unibound-Delta}, the unconditional
recursive inequality becomes
\[
a_t
\leq
\widehat c_t
\left(
\|\widehat K_t\|a_{t-1}
+
\mbE[\widehat\varepsilon_{t-1}]
\right).
\]
\end{remark}

The following theorem quantifies the propagation of model misspecification.

\begin{theorem}\label{thm:R_B_MM}
Let $\widehat{\mathcal I}_t$ and $a_0$ be as in \eqref{eq-4-sost} and \eqref{eq-5-sost}, respectively. For every \(t\geq1\),
\[
\begin{aligned}
\mbE\!\left[
\|\Phi_{Y_{1:t}}(\pi_0)
-
\widehat\Phi_{Y_{1:t}}(\widehat\pi_0)\|_{\mathrm{TV}}
\right]
\leq
a_0
\left(
\prod_{i=1}^{t}\|K_i\|
\right)
\left(
\prod_{i\in\widehat{\mathcal I}_t^c}
(1-\bar{\delta}_i)
\right)
\left(
\prod_{i\in\widehat{\mathcal I}_t}
(1+\|L_i\|)
\right)
\\
\qquad+
\sum_{j=1}^{t}
\varepsilon_{j-1}
\left(
\prod_{i=j+1}^{t}\|K_i\|
\right)
\left(
\prod_{\substack{i\in\widehat{\mathcal I}_t^c\\ i\geq j}}
(1-\bar{\delta}_i)
\right)
\left(
\prod_{\substack{i\in\widehat{\mathcal I}_t\\ i\geq j}}
(1+\|L_i\|)
\right).
\end{aligned}
\]
\end{theorem}

\begin{proof}
For \(0\leq s\leq t\), define
\[
\widehat{\mK}_{s,t}
:=
\begin{cases}
1, & s=t,\\[1mm]
\displaystyle
\prod_{i=s+1}^{t}\bigl[\widehat c_i\|K_i\|\bigr],
& 0\leq s<t.
\end{cases}
\]
Then, for every \(0\leq s<t\)
\beq
\widehat{\mK}_{s,t}
=
\widehat c_t\|K_t\|
\widehat{\mK}_{s,t-1},
\label{eq-6-sost}
\eeq
and we aim at proving by induction that
\[
a_t
\leq
\widehat{\mK}_{0,t}a_0
+
\sum_{j=1}^{t}
\widehat c_j
\widehat{\mK}_{j,t}
\varepsilon_{j-1},
\quad \text{for} \quad t\geq1.
\]

When \(t=1\), \eqref{eq:recursion} readily gives
\[
a_1
\leq
\widehat c_1\|K_1\|a_0
+
\widehat c_1\varepsilon_0
=
\widehat{\mK}_{0,1}a_0
+
\widehat c_1\widehat{\mK}_{1,1}\varepsilon_0.
\]
Let us now assume that the claim holds at time \(t-1\). Using \eqref{eq:recursion} again,
\beqa
a_t
&\leq&
\widehat c_t\|K_t\|a_{t-1}
+
\widehat c_t\varepsilon_{t-1}
\nn \\
&\leq&
\widehat c_t\|K_t\|
\left(
\widehat{\mK}_{0,t-1}a_0
+
\sum_{j=1}^{t-1}
\widehat c_j
\widehat{\mK}_{j,t-1}
\varepsilon_{j-1}
\right)
+
\widehat c_t\varepsilon_{t-1}
\nn \\
&=&
\widehat{\mK}_{0,t}a_0
+
\sum_{j=1}^{t-1}
\widehat c_j
\widehat{\mK}_{j,t}
\varepsilon_{j-1}
+
\widehat c_t
\widehat{\mK}_{t,t}
\varepsilon_{t-1}
\nn \\
&=&
\widehat{\mK}_{0,t}a_0
+
\sum_{j=1}^{t}
\widehat c_j
\widehat{\mK}_{j,t}
\varepsilon_{j-1},
\label{eq-9-sost}
\eeqa
where the second inequality follows from the induction hypothesis and the first equality is a consequence of \eqref{eq-6-sost}. The second equality is straightforward and completes the induction argument.

It remains to express the factors in terms of \(\widehat{\mathcal I}_t\) and \(\bar{\delta}_i\). By definition of \(\widehat c_i\) in \eqref{eq-3-sost} and $\widehat{\mathcal I}_t$ in \eqref{eq-4-sost} we readily obtain
\beq
\widehat{\mK}_{0,t}
=
\prod_{i=1}^{t}\bigl[\widehat c_i\|K_i\|\bigr]
=
\left(
\prod_{i=1}^{t}\|K_i\|
\right)
\left(
\prod_{i\in\widehat{\mathcal I}_t^c}
(1-\bar{\delta}_i)
\right)
\left(
\prod_{i\in\widehat{\mathcal I}_t}
(1+\|L_i\|)
\right).
\label{eq-7-sost}
\eeq
Similarly, for \(1\leq j\leq t\),
\beq
\widehat c_j\widehat{\mK}_{j,t}
=
\widehat c_j
\prod_{i=j+1}^{t}\bigl[\widehat c_i\|K_i\|\bigr]
=
\left(
\prod_{i=j+1}^{t}\|K_i\|
\right)
\left(
\prod_{\substack{i\in\widehat{\mathcal I}_t^c\\ i\geq j}}
(1-\bar{\delta}_i)
\right)
\left(
\prod_{\substack{i\in\widehat{\mathcal I}_t\\ i\geq j}}
(1+\|L_i\|)
\right).
\label{eq-8-sost}
\eeq
Substituting \eqref{eq-7-sost} and \eqref{eq-8-sost} into \eqref{eq-9-sost} completes the proof.
\end{proof}

The first term in Theorem~\ref{thm:R_B_MM} describes the propagation of the
initial discrepancy, whereas the second term accounts for the accumulation
of the perturbations introduced by the misspecified dynamics. A perturbation
generated at time \(j\) is affected by the update at that same time and is
subsequently propagated through all later prediction and update steps.
Consequently, update times belonging to
\(\widehat{\mathcal I}_t\) introduce additional multiplicative factors into
the bounds for both the error already present in the recursion and the
dynamical perturbations introduced up to that time. Thus, as in
Section~\ref{sec3}, non-expansivity need not hold at every update time,
although potentially expansive updates introduce additional factors in the
propagation of the misspecification error, while non-expansive updates actively reduce it through the factor \((1-\bar{\delta}_i)\).

An important particular case is obtained when every update is non-expansive in expectation. Then \(\widehat{\mathcal I}_t=\emptyset\) for every \(t\), and the inequality in Theorem~\ref{thm:R_B_MM} simplifies significantly, as shown below.

\begin{corollary}\label{thm:lim_bound}
Assume that
\(
\widehat\xi_{Y_{1:t-1}}
\in
\mathfrak U_{\mathrm E}(\xi_{Y_{1:t-1}})
\)
for all \(t\geq1\). Let 
\[
C := \sup_{t\geq 1} \big( \|K_t\| (1-\bar{\delta}_t) \big)
\]
denote the joint effective contraction rate of the prediction and update steps, and assume that
\(
C < 1.
\)
Then, for every \(t\geq1\),
\beq
\mbE\left[
\|\Phi_{Y_{1:t}}(\pi_0)
-
\widehat\Phi_{Y_{1:t}}(\widehat\pi_0)\|_{\mathrm{TV}}
\right]
\leq
a_0 e^{-(1-C)t}
+
\varepsilon
\frac{1-C^{\,t}}{1-C},
\label{eq-7.5-sost}
\eeq
where $\varepsilon$ is the time-uniform bound in \eqref{rem:unibound-Delta2}. As a consequence,
\begin{equation}\label{cor:exp}
\limsup_{t\to\infty}
\mbE\!\left[
\|\Phi_{Y_{1:t}}(\pi_0)
-
\widehat\Phi_{Y_{1:t}}(\widehat\pi_0)\|_{\mathrm{TV}}
\right]
\leq
\frac{\varepsilon}{1-C}.
\end{equation}
\end{corollary}

\begin{remark}
The asymptotic bound in \eqref{cor:exp} is no larger than the trivial total
variation bound whenever
\(
C+\varepsilon\leq1,
\)
and is strictly informative whenever
\(
C+\varepsilon<1.
\)
\end{remark}

\begin{proof}
(Of Corollary \ref{thm:lim_bound}) Since
\(
\widehat\xi_{Y_{1:t-1}}
\in
\mathfrak U_{\mathrm E}(\xi_{Y_{1:t-1}})
\)
for every \(t\geq1\), we have
\(
\widehat{\mathcal I}_t=\emptyset
\)
for every \(t\). Hence, all the expansive update factors \((1+\|L_i\|)\) in
Theorem~\ref{thm:R_B_MM} are absent, and therefore
\beq
\begin{aligned}
\mbE\!\left[
\|\Phi_{Y_{1:t}}(\pi_0)
-
\widehat\Phi_{Y_{1:t}}(\widehat\pi_0)\|_{\mathrm{TV}}
\right]
&\leq
a_0\prod_{i=1}^{t}\bigl[\|K_i\|(1-\bar{\delta}_i)\bigr]
\\
&\quad +
\sum_{j=1}^{t}
\left(
\prod_{i=j+1}^{t}\bigl[\|K_i\|(1-\bar{\delta}_i)\bigr]
\right)
\varepsilon_{j-1}.
\end{aligned}
\label{eq-8.5-sost}
\eeq
By definition, \(\|K_i\|(1-\bar{\delta}_i) \leq C < 1\) for every \(i\geq1\). Hence,
\beq
a_0\prod_{i=1}^{t}\bigl[\|K_i\|(1-\bar{\delta}_i)\bigr]
\leq
a_0 C^{\,t}
\leq
a_0e^{-(1-C)t}.
\label{eq-9.5-sost}
\eeq
Moreover, since \(\varepsilon_j\leq\varepsilon\) for every \(j\geq0\),
\beq
\sum_{j=1}^{t}
\left(
\prod_{i=j+1}^{t}\bigl[\|K_i\|(1-\bar{\delta}_i)\bigr]
\right)
\varepsilon_{j-1}
\leq
\varepsilon
\sum_{j=1}^{t}C^{\,t-j}
=
\varepsilon
\sum_{\ell=0}^{t-1}C^{\,\ell}
=
\varepsilon
\frac{1-C^{\,t}}{1-C}.
\label{eq-10-sost}
\eeq
Substituting \eqref{eq-9.5-sost} and \eqref{eq-10-sost} back into \eqref{eq-8.5-sost} yields inequality \eqref{eq-7.5-sost}. Inequality \eqref{cor:exp} is now straightforward.
\end{proof}

Corollary~\ref{thm:lim_bound} shows that, provided that all update steps are non-expansive in expectation, the filtering error can be upper bounded using the joint effective contraction coefficient \(C < 1\) (which tightly couples the contraction of the true kernels \(\|K_t\|\) and the non-expansive updates \((1-\bar{\delta}_t)\)) and the uniform misspecification bound \(\varepsilon\) alone.

%
\subsection{Summable misspecification}
\label{sscec:Decay_missp}

We now turn to almost sure convergence. For this purpose, we impose the
stronger conditional non-expansivity condition $\widehat \xi_{Y_{1:t-1}} \in \mathfrak U(\xi_{Y_{1:t-1}})$ along the entire sequence $\xi_{Y_{1:t}}$. 
This restores the conditional structure underlying
Remark~\ref{rem:super_Martingale}. In the misspecified setting, however,
the discrepancy between the true and perturbed dynamics introduces an
additive forcing term into the conditional error recursion. The resulting
process therefore has the standard almost-supermartingale form rather
than the exact supermartingale structure obtained in Section~\ref{sec3}.

\begin{remark}\label{rem:almost_SM}
Define
\[
\widehat E_t
:=
\|
\Phi_{Y_{1:t}}(\pi_0)
-
\widehat\Phi_{Y_{1:t}}(\widehat\pi_0)
\|_{\mathrm{TV}},
\quad t\geq1,
\]
with
\(
\widehat E_0
:=
\|\pi_0-\widehat\pi_0\|_{\mathrm{TV}}
\), and note that $a_t=\mathbb E[\widehat E_t ]$ with the expectation taken over the law of $Y_{1:t}$. If we assume that
\(
\widehat\xi_{Y_{1:t-1}}
\in
\mathfrak U(\xi_{Y_{1:t-1}})
\;
\text{for all $t\ge 1$}
\),
then, using Lemma~\ref{lemma:ExFilter_onestep}, we have
\beq
\mathbb E[
\widehat E_t
|\,
\mathcal G_{t-1}
]
\leq
(1-\delta_t)\|K_t\|\widehat E_{t-1}
+
(1-\delta_t)\varepsilon_{t-1},
\label{eq:general_onestep}
\eeq
where \(\delta_t := \delta(\widehat\xi_{Y_{1:t-1}})\) is the conditional contraction rate. Thus, the error process has the standard non-negative almost-supermartingale form in the sense of \cite{robbins1971convergence}.
\end{remark}

The term $(1-\delta_t)\| K_t \| \widehat E_{t-1}$ represents the combined dissipative effect of the predictive dynamics and update step, whereas the additive forcing introduced by model mismatch is $(1-\delta_t)\varepsilon_{t-1}$. Notice that since $0 \le 1-\delta_t \le 1$, this forcing term is upper bounded by $\varepsilon_{t-1}$. Unlike the expected-error estimates of Theorem~\ref{thm:R_B_MM}, the almost-supermartingale setting requires conditional non-expansivity at every update time, namely $\widehat \xi_{Y_{1:t-1}} \in \mathfrak U(\xi_{Y_{1:t-1}})$.

We now consider the case in which the cumulative forcing remains finite a.s.. Specifically, the almost-supermartingale convergence theorem requires summability of the effective forcing sequence, i.e., $\sum_{t=1}^{\infty}(1-\delta_t)\varepsilon_{t-1} < \infty$ a.s. Since $1-\delta_t \le 1$, a convenient deterministic sufficient condition is that the sequence of predictive discrepancies $\{\varepsilon_t\}_{t\ge 0}$ itself be summable, $\sum_{t=0}^{\infty}\varepsilon_t < \infty$. Alternatively, effective summability can also hold when $\sum_{t=0}^{\infty}\varepsilon_t = \infty$, provided the update steps supply sufficient contraction so that $1-\delta_t$ decays fast enough.

Let
\(
\{\varepsilon_t\}_{t\geq0}
\)
and
\(
\{\widehat\varepsilon_t\}_{t\geq0}
\)
denote the perturbation sequences defined in
\eqref{rem:unibound-Delta2} and \eqref{rem:unibound-Delta}, respectively.
The sequence \(\{\varepsilon_t\}_{t\geq0}\) is deterministic and satisfies
\(0\leq\varepsilon_t\leq1\), whereas
\(\{\widehat\varepsilon_t\}_{t\geq0}\) is \(\mathcal G\)-adapted and satisfies
\(
0\leq\widehat\varepsilon_t\leq\varepsilon_t\leq1.
\)

Throughout this subsection, we state the results for the sequence
\(\{\varepsilon_t\}_{t\geq0}\). The corresponding trajectory-dependent
version follows by replacing
\(
(1-\delta_t)\|K_t\|
\)
and
\(
\varepsilon_t
\)
with
\(
(1-\delta_t)\|\widehat K_t\|
\)
and
\(
\widehat\varepsilon_t,
\)
respectively, and using Proposition~\ref{PPofE} instead of
Proposition~\ref{PPofE2}.

The following theorem establishes almost-sure coalescence under summable effective misspecification. Under conditional non-expansivity, summability of the effective perturbations \((1-\delta_t)\varepsilon_{t-1}\) ensures that only a finite amount of forcing is introduced into the recursion, while the divergence of the accumulated contraction, jointly induced by the predictive operator and the Bayesian update, guarantees that the limiting filtering error vanishes.


\begin{theorem}[Asymptotic stability under accumulated contraction]
\label{thm:AS_miss_subsequence}
Let $\{\widehat E_t\}_{t\geq0}$ be as in
Remark~\ref{rem:almost_SM}, and suppose that
$\widehat\xi_{Y_{1:t-1}}\in
\mathfrak U(\xi_{Y_{1:t-1}})$ almost surely
for every $t\geq1$. Assume that
\[
\sum_{t=1}^{\infty}
(1-\delta_t)\varepsilon_{t-1}<\infty
\qquad\text{a.s.}
\]
and
\begin{equation}
\label{eq:accumulated_contraction}
\sum_{t=1}^{\infty}
\left[1-(1-\delta_t)\|K_t\|\right]
=\infty
\qquad\text{a.s.}.
\end{equation}
Then
\(
\widehat E_t\xrightarrow{\mathrm{a.s.}}0,
\;
\widehat E_t\xrightarrow{L^1}0,
\;
\operatorname{Var}(\widehat E_t)\to0.
\)
\end{theorem}

\begin{proof}
Set
\[
\rho_t:=(1-\delta_t)\|K_t\|
\quad \text{and} \quad
\eta_t:=(1-\delta_t)\varepsilon_{t-1}.
\]
Then, using \eqref{eq:general_onestep},
\[
\mathbb E[\widehat E_t\mid\mathcal G_{t-1}]
\leq
\widehat E_{t-1}
-(1-\rho_t)\widehat E_{t-1}
+\eta_t
\]
and, since $0\leq\rho_t\leq1$, $0\leq\widehat E_t\leq1$, $0\leq \eta_t\leq1$ and $\sum_{t=1}^{\infty}\eta_t<\infty$ a.s., the Robbins--Siegmund theorem (see Theorem~\ref{thm:robbins_siegmund}) yields
\begin{equation}
\label{eq:RS_accumulated}
\widehat E_t\longrightarrow\widehat E_\infty,
\qquad
\sum_{t=1}^{\infty}
(1-\rho_t)\widehat E_{t-1}<\infty,
\qquad\text{a.s.}
\end{equation}
Suppose that $\widehat E_\infty>0$. Then, for sufficiently large $T$, $\widehat E_{t-1}\geq\widehat E_\infty/2$, $\forall t\geq T$, and hence
\[
\sum_{t=1}^{\infty}
(1-\rho_t)\widehat E_{t-1}
\geq
\frac{\widehat E_\infty}{2}
\sum_{t=T}^{\infty}(1-\rho_t)
=\infty,
\]
contradicting \eqref{eq:RS_accumulated}. Therefore, $\widehat E_\infty=0$ a.s. Since $0\leq\widehat E_t\leq1$, Lemma~\ref{lem:Bounded_Conv_Equiv} establishes the remaining claims.
\end{proof}

\begin{remark}[Interpretation]
\label{rem:summable_particular_case}
Observe that
\begin{equation}
\label{eq:contraction_decomposition}
1-(1-\delta_t)\|K_t\|
=
(1-\|K_t\|)+\delta_t\|K_t\|.
\end{equation}
The first term represents contraction due to
the predictive operator, while the second
accounts for the additional contraction induced
by the Bayesian update. Both terms are
non-negative, and therefore
\eqref{eq:accumulated_contraction} requires only
that their combined contributions accumulate
indefinitely. In particular, neither the
prediction nor the update needs to provide
a uniformly positive contraction rate.
\end{remark}

\begin{remark}[Sufficient conditions]
The assumptions of Theorem~\ref{thm:AS_miss_subsequence}
can be verified through simpler, albeit stronger,
sufficient conditions. Indeed, since
$0\leq1-\delta_t\leq1$, the summability
condition on the model misspecification holds
whenever
\[
\sum_{t=0}^{\infty}\varepsilon_t<\infty
\qquad\text{a.s.}
\]
Similarly, the accumulated contraction condition
\eqref{eq:accumulated_contraction} is satisfied if
\[
\liminf_{t\to\infty}
(1-\delta_t)\|K_t\|<1
\qquad\text{a.s.}
\]
The latter requires the effective contraction
factor to remain bounded away from $1$ infinitely
often.
Neither condition is necessary for the theorem
to hold.
\end{remark}

\begin{remark}[Trajectory-dependent misspecification]
An analogous result holds when considering the
contraction coefficient of the misspecified
transition operator $\widehat K_t$ and the
trajectory-dependent perturbation
$\widehat\varepsilon_{t-1}$.
Since $\delta_t$ and $\widehat\varepsilon_{t-1}$
are bounded, $\mathcal G_{t-1}$-measurable
random variables, their product
$(1-\delta_t)\widehat\varepsilon_{t-1}$
is non-negative, adapted, and integrable.
Thus, the Robbins--Siegmund theorem applies
under the conditions
\[
\sum_{t=1}^{\infty}
(1-\delta_t)\widehat\varepsilon_{t-1}<\infty,
\qquad
\sum_{t=1}^{\infty}
\left[1-(1-\delta_t)\|\widehat K_t\|\right]
=\infty,
\quad\text{a.s.}
\]
The first condition holds, in particular, if
$\sum_{t=0}^{\infty}\widehat\varepsilon_t<\infty$
almost surely.

Hence, the true and misspecified transition
operators may remain substantially different
globally. Asymptotic stability only requires
that their effective discrepancy, evaluated
along the optimal filtering trajectory and
weighted by the contraction of the Bayesian
update, be summable, together with sufficient
accumulated contraction.
\end{remark}

Recall the filtering error process \(E_t\) defined in
\eqref{eq:filtering-error-process}. While
Theorem~\ref{thm:AS_miss_subsequence} accounts for both prior and dynamical
misspecification, the following corollary recovers the case of an
alternative initialization of the true filtering recursion considered in
Section~\ref{sec3}.

\begin{corollary}\label{cor:hat_to_unhat_error}
Let $\{E_t\}_{t\geq0}$ be the filtering error process
defined in \eqref{eq:filtering-error-process}.
Suppose that
\(
K_t(\Phi_{Y_{1:t-1}}(\widehat\pi_0))
\in\mathfrak U(\xi_{Y_{1:t-1}})
\;\text{a.s., for every }t\geq1,
\)
and that
\[
\sum_{t=1}^{\infty}
\left[1-(1-\delta_t)\|K_t\|\right]
=\infty
\qquad\text{a.s.}
\]
Then
\(
E_t\xrightarrow{\mathrm{a.s.}}0,
\;
E_t\xrightarrow{L^1}0,
\;
\operatorname{Var}(E_t)\to0.
\)
\end{corollary}

\begin{proof}
Taking $\widehat K_t=K_t$ for every $t\geq1$
gives $\widehat\Phi_{Y_{1:t}}(\widehat\pi_0)
=\Phi_{Y_{1:t}}(\widehat\pi_0)$, and hence
$\widehat E_t=E_t$.
Moreover, $\varepsilon_t=0$ for all $t\geq0$,
so the effective misspecification is
identically zero. The result follows directly
from Theorem~\ref{thm:AS_miss_subsequence}.
\end{proof}

%


\section{Examples}\label{sec5}

This section illustrates the main conditions underlying the stability and misspecification results of Sections~\ref{sec3} and~\ref{sec4}. We focus on the two ingredients identified in the preceding analysis: the contraction induced by the update operators (quantified by the contraction rates $\delta$ and $\bar{\delta}$) and membership in the class $\mathfrak{U}(\xi_{Y_{1:t-1}}),\,t\ge1$, of probability measures that are non-expansive w.r.t. the true predictive measures, alongside the contraction of transition kernels when present. First, we consider a scalar linear--Gaussian state-space model and explicitly show how the non-expansivity condition defining $\mathfrak{U}(\xi_{Y_{1:t-1}})$ can be verified for a class of Gaussian probability measures and preserved through the Kalman prediction and update equations. This provides a concrete setting in which the admissibility assumption introduced in Section~\ref{sec3} can be established analytically. Second, we consider several classes of transition kernels for which the contraction coefficients of the predictive operators can be controlled explicitly through Doeblin-type minorization conditions.

%
\subsection{Scalar linear--Gaussian case}
\label{sec5:Linear_Gaussian}

We identify a sufficient condition for Gaussian predictive
measures to satisfy the update non-expansivity condition
of Section~3. We then show that this condition is preserved
by the 1-dimensional Kalman filter recursion.

%
\subsubsection{Gaussian predictive laws}

We proceed with a one-step analysis first. To this end, choose some \(t\geq1\) and write
\[
\xi:=\xi_{Y_{1:t-1}}=\mathcal N(m_t,P_t),
\quad
\nu:=\nu_{Y_{1:t-1}}
=\mathcal N(\widehat m_t,\widehat P_t)
\]
for conciseness, where the parameters $m_t, \widehat m_t, P_t$ and $\widehat P_t$ are \(\mathcal G_{t-1}\)-measurable and \(P_t,\widehat P_t>0\) a.s.. We also assume the linear-Gaussian observation model
\beq\label{SLG_obs_model}
Y_t=H_tX_t+V_t,
\quad \text{with} V_t\sim\mathcal N(0,R_t),
\quad \text{and} \quad
R_t>0.
\eeq
The corresponding conditional observation laws are
\[
L_t(\xi)=\mathcal N(H_tm_t,S_t)
\quad \text{and} \quad
L_t(\nu)=\mathcal N(H_t\widehat m_t,\widehat S_t),
\]
where
\[
S_t=H_t^2P_t+R_t
\quad \text{and} \quad
\widehat S_t=H_t^2\widehat P_t+R_t.
\]

Recall that, for a fixed observation \(y_t\), we denote \(g_{y_t}(x):=g_t(y_t\mid x)\) and denote the corresponding update operator by \(U_{y_t}\). If we write
\[
h_t:=\frac{\mathrm d\nu}{\mathrm d\xi},
\quad
r_{y_t}:=\frac{\nu(g_{y_t})}{\xi(g_{y_t})},
\quad
A_{y_t}:=\{h_t\leq r_{y_t}\}
\quad \text{and} \quad
q_{y_t}:=U_{y_t}(\nu)(A_{y_t}),
\]
then it can be shown that
\[
\frac{\mathrm dU_{y_t}(\nu)}
     {\mathrm dU_{y_t}(\xi)}(x)
=
\frac{h_t(x)}{r_{y_t}}
\]
and the set \(A_{y_t}\) is a positive Hahn set of \(U_{y_t}(\xi)-U_{y_t}(\nu)\). Likewise, \(\widetilde A_t:=\{h_t\leq1\}\) is a positive Hahn set of \(\xi-\nu\).

Using the notation in Eq. \eqref{eqlong} for the correction terms, we obtain
\beq
\mT_{{\rm cov},t}
=
\operatorname{Cov}^{L_t(\xi)}
\left(r_{Y_t},q_{Y_t}\right)
\quad \text{and} \quad
\mT_{{\rm rem},t}
=
\mathbb E^{L_t(\xi)}
\left[
  \frac{
    |\xi-\nu|
    \left(
      g_{Y_t}\mathbf 1_{A_{Y_t}\triangle\widetilde A_t}
    \right)
  }{
    \xi(g_{Y_t})
  }
\right]
\geq0.
\eeq
The dependence of $\mT_{{\rm cov},t}$ and $\mT_{{\rm rem},t}$ on the observation history
is left implicit. Theorem~\ref{lemma:ExUpd} gives, \(\mathbb P\)-a.s.,
\[
\mathbb E^{L_t(\xi)}
\left[
  \|U_{Y_t}(\xi)-U_{Y_t}(\nu)\|_{\mathrm{TV}}
\right]
=
\|\xi-\nu\|_{\mathrm{TV}}
+\mT_{{\rm cov},t}-\mT_{{\rm rem},t}.
\]

\begin{proposition}
\label{prop:scalar-gaussian-nonexpansion}
Consider the scalar linear--Gaussian observation model
\eqref{SLG_obs_model}, with the Gaussian predictive
measures \(\xi\) and \(\nu\) defined above.
If \(\widehat P_t\geq P_t\) a.s., then
\[
\mT_{{\rm cov},t}
=
\operatorname{Cov}^{L_t(\xi)}
\left(r_{Y_t},q_{Y_t}\right)
\leq0
\qquad\mathbb P\text{-a.s.}
\]
As a consequence, \(\nu_{Y_{1:t-1}} \in \mathfrak U(\xi_{Y_{1:t-1}})\), and the contraction rates defined in Definition~\ref{def:contraction-rates} satisfy \(\delta(\nu_{Y_{1:t-1}}) \geq 0\) a.s.. In particular,
\beq
\mathbb E^{L_t(\xi)}
\left[
  \|U_{Y_t}(\xi)-U_{Y_t}(\nu)\|_{\mathrm{TV}}
\right]
\leq
(1-\delta(\nu_{Y_{1:t-1}})) \|\xi-\nu\|_{\mathrm{TV}}
\leq
\|\xi-\nu\|_{\mathrm{TV}}
\qquad\mathbb P\text{-a.s.}
\eeq
\end{proposition}

A proof, with the time index suppressed, is provided in Appendix~\ref{Ap_P5.1}. The proposition imposes no restriction on the predictive means \(m_t\) and \(\widehat m_t\).

\begin{remark}
When $\widehat P_t=P_t$, Proposition~\ref{prop:scalar-gaussian-nonexpansion} extends to multivariate Gaussian measures. In this case, the posterior covariance matrices coincide, and the difference between the posterior means is
$$
\widehat m_t^+-m_t^+
=(I-G_tH_t)(\widehat m_t-m_t),
$$
where $G_t$ denotes the common Kalman gain. A direct calculation shows that the positive Hahn set can be chosen as a half-space whose probability under $U_{y_t}(\nu)$ is independent of $y_t$. Consequently, $q_{y_t}$ is constant under $L_t(\xi)$, and the covariance term vanishes.
\end{remark}
When \(\widehat P_t<P_t\), the variance ordering alone
does not determine whether the update is non-expansive.
The covariance term \(\mT_{{\rm cov},t}\) may be positive,
but membership in \(\mathfrak U(\xi)\) depends on whether
it is compensated by \(\mT_{{\rm rem},t}\).

To illustrate this distinction at a fixed time \(t\), take
\[
m_t=0,
\quad
P_t=1,
\quad \text{and} \quad
H_t=R_t=1,
\]
and define the expected net change in total variation as
\[
\Delta_{\mathrm{TV},t}
:=
\mathbb E^{L_t(\xi)}
\left[
  \|U_{Y_t}(\xi)-U_{Y_t}(\nu)\|_{\mathrm{TV}}
\right]
-
\|\xi-\nu\|_{\mathrm{TV}}.
\]
The exact update identity gives 
$
\Delta_{\mathrm{TV},t}
=
\mT_{{\rm cov},t}-\mT_{{\rm rem},t}
$. Thus \(\Delta_{\mathrm{TV},t}\leq0\) is precisely the non-expansivity condition required for membership in \(\mathfrak U(\xi)\), under which the conditional contraction rate is well-defined as \(\delta_t = -\Delta_{\mathrm{TV},t} / \|\xi - \nu\|_{\mathrm{TV}} \in [0, 1]\).

Table~\ref{tab:scalar-gaussian-update} reports the numerical values obtained by varying the alternative variance or mean. For \(\widehat P_t\geq P_t\), the covariance is non-positive, as established in Proposition~\ref{prop:scalar-gaussian-nonexpansion}, guaranteeing \(\Delta_{\mathrm{TV},t} \le 0\). For the nonidentical pairs displayed in Table~\ref{tab:scalar-gaussian-update} that satisfy \(\widehat P_t\geq P_t\), with \(H_t=1\), the updates contract strictly (\(\delta_t>0\)). In general, this variance ordering guarantees non-expansivity; strictness additionally requires \(H_t\ne0\) and \(\xi\ne\nu\). For \(\widehat P_t<P_t\), the table contains both signs of \(\Delta_{\mathrm{TV},t}\): a positive covariance can either exceed the remainder (yielding expansion \(\Delta_{\mathrm{TV},t} > 0\), so \(\nu \notin \mathfrak U(\xi)\)) or be compensated by it (yielding non-expansivity \(\Delta_{\mathrm{TV},t} \le 0\), so \(\delta_t \ge 0\)). These examples illustrate that \(\widehat P_t\geq P_t\) is sufficient, but not necessary, for membership in
\(\mathfrak U(\xi)\).
\begin{table}[htb]
\centering
\small
\setlength{\tabcolsep}{3pt}
\caption{
One-step Gaussian updates with \(m_t=0\), \(P_t=1\), and \(H_t=R_t=1\). The expected change in total variation is \(\Delta_{\mathrm{TV},t}=\mT_{{\rm cov},t}-\mT_{{\rm rem},t}\): positive values indicate expansion (\(\nu \notin \mathfrak U(\xi)\)), while non-positive values indicate non-expansivity (\(\nu \in \mathfrak U(\xi)\) with rate \(\delta_t = -\Delta_{\mathrm{TV},t}/\|\xi-\nu\|_{\mathrm{TV}}\)).
}
\label{tab:scalar-gaussian-update}
\begin{minipage}[t]{0.46\linewidth}
\centering
\textit{(a) Different variances, fixed \(\widehat m_t=-2\)}
\par\smallskip
\begin{tabular*}{\linewidth}{@{\extracolsep{\fill}}rrrr@{}}
\toprule
\(\widehat P_t\) &
\(\mT_{{\rm cov},t}\) &
\(\mT_{{\rm rem},t}\) &
\(\Delta_{\mathrm{TV},t}\)\\
\midrule
0.001 &  0.003883 & 0.002740 &  0.001143\\
0.10  &  0.044856 & 0.040484 &  0.004372\\
0.25  &  0.057578 & 0.074196 & -0.016618\\
0.50  &  0.050675 & 0.119086 & -0.068411\\
1.00  &  0        & 0.162190 & -0.162190\\
2.00  & -0.098341 & 0.164006 & -0.262347\\
\bottomrule
\end{tabular*}
\end{minipage}\hfill%
\begin{minipage}[t]{0.46\linewidth}
\centering
\textit{(b) Different means, fixed \(\widehat P_t=0.5<P_t\)}
\par\smallskip
\begin{tabular*}{\linewidth}{@{\extracolsep{\fill}}rrrr@{}}
\toprule
\(\widehat m_t\) &
\(\mT_{{\rm cov},t}\) &
\(\mT_{{\rm rem},t}\) &
\(\Delta_{\mathrm{TV},t}\)\\
\midrule
0.00 & 0.000732 & 0.011226 & -0.010494\\
0.50 & 0.004412 & 0.037484 & -0.033071\\
1.00 & 0.019359 & 0.078407 & -0.059048\\
2.00 & 0.050675 & 0.119086 & -0.068411\\
3.00 & 0.051184 & 0.098553 & -0.047370\\ 
4.00 & 0.035608 & 0.057708 & -0.022100\\
\bottomrule
\end{tabular*}
\end{minipage}
\end{table}

%
\subsubsection{1-dimensional Kalman filter}

The sufficient variance ordering in Proposition~\ref{prop:scalar-gaussian-nonexpansion} is preserved when both filters use the same scalar linear--Gaussian dynamics and observation model. This yields an explicit condition on the initial distributions that guarantees membership of $\mathfrak U(\xi_{Y_{1:t-1}})$ at every time step.

Consider the state dynamics
\[
X_t=A_tX_{t-1}+W_t,
\quad \text{with} \quad 
W_t\sim\mathcal N(0,Q_t),
\quad \text{and} \quad
Q_t\geq0,
\]
together with the observation model \eqref{SLG_obs_model}. We assume that the r.v.s in the sequences \(\{W_t\}_{t\geq1}\) and \(\{V_t\}_{t\geq1}\) are independent across time, that the two sequences are mutually independent, and that both are independent of the initial state \(X_0\). The two filters start from Gaussian distributions
\[
\pi_0=\mathcal N(m_0^+,P_0^+)
\quad \text{and} \quad
\widehat\pi_0
=\mathcal N(\widehat m_0^+,\widehat P_0^+).
\]
For \(t\geq1\), the Kalman filter yields $\xi_{Y_{1:t-1}} = \mN(m_t,P_t)$ and $\pi_{Y_{1:t}} = \mN(m_t^+,P_t^+)$ for the true model, and $\widehat \xi_{Y_{1:t-1}} = \mN(\widehat m_t, \widehat P_t)$ and $\widehat\pi_{Y_{1:t}} = \mN(\widehat m_t^+, \widehat P_t^+)$ for the misspecified model. 


\begin{proposition}\label{prop:scalar-gaussian-recursion}
If \(\widehat P_0^+ \geq P_0^+ > 0\) and \(A_t^2+Q_t>0\) for every \(t\geq1\), then, for every \(t\geq1\),
\[
\widehat P_t\geq P_t
\qquad\text{and}\qquad
\widehat\xi_{Y_{1:t-1}}
\in
\mathfrak U(\xi_{Y_{1:t-1}})
\quad\mathbb P\text{-a.s.}
\]
Consequently, the conditional update contraction rate satisfies \(\delta_t := \delta(\widehat\xi_{Y_{1:t-1}}) \geq 0\) \(\mathbb P\)-a.s., and
\[
\mathbb E\left[
  \|\pi_{Y_{1:t}}-\widehat\pi_{Y_{1:t}}\|_{\mathrm{TV}}
  \,\middle|\,\mathcal G_{t-1}
\right]
\leq
(1-\delta_t) \|\xi_{Y_{1:t-1}}-\widehat\xi_{Y_{1:t-1}}\|_{\mathrm{TV}}
\quad\mathbb P\text{-a.s.}
\]
\end{proposition}

\begin{proof}
For the true and misspecified models, the variance recursions are, respectively, 
\[
\begin{aligned}
P_t&=A_t^2P_{t-1}^++Q_t,
&
P_t^+&=\Gamma_t(P_t), \quad \text{and}\\
\widehat P_t&=A_t^2\widehat P_{t-1}^++Q_t,
&
\widehat P_t^+&=\Gamma_t(\widehat P_t),
\end{aligned}
\]
where
\[
\Gamma_t(P):=\frac{PR_t}{H_t^2P+R_t}.
\]
Since \(A_t^2+Q_t>0\) and \(R_t>0\), it follows that $\widehat P_t, P_t, \widehat P_t^+$ and $P_t^+$ are strictly positive, and we note that the derivative of $\Gamma_t$ w.r.t. $P$ is 
\[
\Gamma_t'(P)
=
\frac{R_t^2}{(H_t^2P+R_t)^2}>0.
\]
Hence, both steps (prediction and update) preserve the variance ordering, and a simple induction argument gives
\[
\widehat P_0^+\geq P_0^+>0
\quad\Longrightarrow\quad
\widehat P_t\geq P_t
\quad\text{for every $t\geq 1$}.
\]

Conditionally on the observation history, the predictive measures $\xi_{Y_{1:t-1}}$ and $\widehat \xi_{Y_{1:t-1}}$ are Gaussian with variances $P_t$ and $\widehat P_t$, respectively. Proposition~\ref{prop:scalar-gaussian-nonexpansion} holds irrespective of the means, ensuring \(\widehat\xi_{Y_{1:t-1}} \in \mathfrak U(\xi_{Y_{1:t-1}})\) and \(\delta_t \geq 0\). Applying the conditional bound from Lemma~\ref{lemma:ExFilter_onestep} completes the proof.
\end{proof}

Thus, in the scalar linear--Gaussian model, the ordering $\widehat P_0^+\geq P_0^+>0$ provides an explicit sufficient condition under which the property
\(
\widehat\xi_{Y_{1:t-1}}
\in
\mathfrak U(\xi_{Y_{1:t-1}})
\)
is preserved along the entire filtering recursion with non-negative update contraction rates \(\delta_t \geq 0\). This gives a concrete analytical verification of one of the central assumptions underlying the stability analysis of Section~\ref{sec3} and Section~\ref{sec4}.

%
\subsection{Truncated kernels}
\label{ssec:truncated_kernels}

As established in the stability and misspecification bounds of Sections~\ref{sec3} and~\ref{sec4}, strict contraction of the transition kernels (i.e., $\|K_t\| < 1$) is not strictly necessary for filter stability, provided that the observation update steps supply sufficient contraction\footnote{For example, with $\|K\| = 1$ and every step non-expansive in expectation, $\sum_t \bar \delta_t = \infty$ ensures convergence of the expected error.}. Nevertheless, when the underlying predictive dynamics are themselves contractive, this predictive contraction acts jointly with the update step to yield tighter overall contraction constants such as $C = \sup_{t \ge 1} \|K_t\|(1-\bar{\delta}_t) < 1$ in Corollary~\ref{thm:lim_bound}, thereby accelerating the exponential rate of convergence and sharpening the asymptotic error bounds. In this and the following subsection, we analyze conditions on the transition dynamics that guarantee explicit, analytically verifiable predictive contraction bounds $\|K_t\| \le 1-\kappa < 1$.

Assume that the transition kernel can be characterized by way of a transition pdf. On a compact state space, continuity and strict positivity of this density imply a Doeblin minorization condition, and hence contraction of the associated predictive operator.

\begin{proposition}\label{prop:Doeblin_compact}
Let \(\mathcal X\subset\mathbb R^d\) be compact such that \(\int_\mX \sd x'>0\). Assume that kernel \(K\) admits a continuous, strictly positive density \(p(x,x')\) w.r.t. Lebesgue measure, that is,
\[
K(x,\mathrm dx')
=
p(x,x') \mathrm dx'.
\]
Then \(K\) satisfies a Doeblin minorization condition w.r.t. the uniform probability measure $\sf u$ on
\((\mathcal X,\mathcal B(\mathcal X))\). Consequently,
\(
\|K\|
\leq
1-\kappa
<1,
\)
where
\beq
\kappa
:=
m_p\,\int_\mX \sd x'
\in(0,1] 
\quad \text{and} \quad 
m_p
:=
\min_{(x,x')\in\mathcal X\times\mathcal X}
p(x,x')
>0.
\eeq
\end{proposition}
\begin{proof}
Since \(\mathcal X\times\mathcal X\) is compact and \(p\) is continuous and
strictly positive, the minimum
\(
m_p
\)
is well defined and satisfies \(m_p>0\).
Hence, for every \(x\in\mathcal X\) and every
\(A\in\mathcal B(\mathcal X)\),
\[
K(x,A)
=
\int_A p(x,x')\mathrm dx'
\geq
m_p\,\int_A \sd x'.
\]
Moreover, since
\(
{\sf u}(A)
=
\dfrac{\int_A \sd x'}
{\int_\mX \sd x'},
\)
it follows that
\(
K(x,A)
\geq
m_p\,\int_\mX \sd x'\,{\sf u}(A)
=
\kappa\,{\sf u}(A).
\)
Since \(K(x,\mathcal X)=\int_{\mX} p(x,x') \sd x' = 1\), necessarily \(\kappa\leq1\).
Thus \(K\) satisfies a Doeblin minorization condition with constant
\(\kappa\). 
The estimate
\(
\|K\|
\leq
1-\kappa
\)
then follows from, e.g., \cite[Lemma~4.3.13, p.~98]{Cappe05}.
\end{proof}

As an example, we may consider a Gaussian transition kernel restricted to the compact state space \(\mX\). Proposition~\ref{prop:Doeblin_compact} yields a Doeblin minorization condition $\|K\|\leq 1-\kappa$, and we can derive an explicit positive lower bound for the constant \(\kappa\).

To see this, let \(\mX\subset\mathbb R^d\) be compact with \(\int_\mX \sd x'>0\) and consider the Gaussian kernel
\(
K(x,\mathrm dx')
=
p(x,x')\mathrm dx'\)
where
\(
p(x,x')
=
\frac{1}{Z(x)}
\exp\!\left(
-\frac12
(x'-h(x))^\top
\Sigma^{-1}
(x'-h(x))
\right),
\)
with continuous transition function \(h:\mX\to\mX\), \(\Sigma\) positive definite, and
\beq
Z(x)
=
\int_{\mX}
\exp\!\left(
-\frac12
(z-h(x))^\top
\Sigma^{-1}
(z-h(x))
\right)
\mathrm dz.
\label{eq-1-star}
\eeq
Since \(\mX\) is compact, there exists
\(\rho>0\) such that
\(
\mX
\subseteq
\overline B(0,\rho),
\) where
\(
\overline B(0,\rho)
:=
\{x\in\mathbb R^d:\|x\|\le \rho\}
\).
In addition, since \(x',h(x)\in\mX\) we have
$$
\|\Sigma^{-1/2}(x'-h(x))\|
\leq
2\rho \|\Sigma^{-1/2}\|\leq 2\rho/ \sqrt{\lambda_{\rm min}(\Sigma)},
$$
where $\Sigma^{1/2}$ is the unique symmetric square root of $\Sigma$ and $\lambda_{\rm min}(\Sigma)$ is the smallest eigenvalue of $\Sigma$.  
Moreover, the exponential function in \eqref{eq-1-star} is bounded by \(1\), so
\(
Z(x)
\leq
\int_\mX \sd x'.
\)
Then
\[
m_p
=
\min_{(x,x')\in\mX\times\mX}
p(x,x')
\geq
\frac{1}
{\int_\mX \sd x'}
\exp\!\left(
-\frac{2\rho^2}
{\lambda_{\min}(\Sigma)}
\right).
\]
Therefore
\(
\kappa
=
m_p\,\int_\mX \sd x'
\geq
\exp\!\left(
-\frac{2\rho^2}
{\lambda_{\min}(\Sigma)}
\right).
\)

%
\subsection{Bounded transition functions}\label{sec:state_dynamics_model}

We next consider SSMs whose state dynamics are given by a uniformly bounded transition map perturbed by additive noise. The associated transition kernels have densities in an elliptically symmetric family. The objective is to show that boundedness of the transition function yields a Doeblin-type minorization condition, and hence contraction through the associated Dobrushin coefficients. Unlike above, the state space \(\mX\) need not be compact.

Assume that the transition kernel $K$ admits a density of the form
\begin{equation}
p(x, x')
=
\phi\!\big(
\psi(
\|x'-h(x)\|_{\Sigma^{-1}}
)
\big),
\label{eq-h-map}
\end{equation}
where
\(
\|z\|_{\Sigma^{-1}}
:= \| \Sigma^{-1/2} z \|
\),
\(\Sigma\in\mbR^{d_x\times d_x}\) is a symmetric positive definite scale matrix, \(\phi:\mbR\to(0,\infty)\) is non-increasing, and
\(
\psi:[0,\infty)\to\mbR
\)
is non-decreasing. Densities of this form belong to the family of elliptically symmetric densities \cite{delmas2024elliptically} and Gaussian kernels are a particular case. 

Further assume that the transition map
\(
h:\mX\to\mX
\)
in \eqref{eq-h-map} is uniformly bounded in the Euclidean norm, namely,
\begin{equation}\label{ineq:h_bound}
M
:=
\sup_{x\in\mX}
\|h(x)\|
<
\infty.
\end{equation}
It is straightforward to see that
\[
s_\Sigma
:=\sup_{x\in\mX}\|h(x)\|_{\Sigma^{-1}}
\leq
b_\Sigma:=\frac{M}{\sqrt{\lambda_{\min}(\Sigma)}},
\]
where \(s_\Sigma\) is the exact supremum in the weighted norm and \(b_\Sigma\) is a simpler upper bound obtained directly from \eqref{ineq:h_bound}.

\begin{theorem}\label{thm:Dobrushin_Doeblin}
Assume that the boundedness condition \eqref{ineq:h_bound} holds, and that the transition kernels belong to the class of elliptically symmetric transition families introduced above. Then  
\(
\|K\|
\leq
1-\kappa,
\) where
\begin{equation}\label{eq:kappa}
\kappa
:=
\int_{\mX}
\phi\!\big(
\psi(
\|x\|_{\Sigma^{-1}}+s_\Sigma
)
\big)\,
\sd x.
\end{equation}
\end{theorem}
See Appendix~\ref{ap:proof_thm:Dobrushin_Doeblin} for a proof.



The kernels indicated in Tables \ref{t-exp-type} and \ref{t-pol-type} belong to the class of elliptically symmetric transition kernels considered above. The minorization constant \(\kappa\) in \eqref{eq:kappa} quantifies the contraction strength of the dynamics. For several radial profiles it can be evaluated explicitly, as given by the following result.

\begin{table}
\caption{Exponential-type radial profile. Let
\(
\phi(v)=C\,\mathrm e^{-v},
\)
where \(C>0\) is a normalization constant.}
\label{t-exp-type}

\begin{center}
\begin{tabular}{ll}
\hline
\textbf{Transition model} & \textbf{Radial deformation \(\psi(r)\)} \\
\hline
Gaussian
& \(\tfrac12 r^2\) \\[0.4em]

Generalized Gaussian / exponential power
& \(r^\beta,\qquad \beta>0\) \\[0.4em]

Laplace
& \(r\) \\[0.4em]

\hline
\end{tabular}
\end{center}

\end{table}

\begin{table}
\caption{Polynomial heavy-tailed radial profile. Let
$
\phi(v)=\left\{
	\begin{array}{ll}
	C\,v^{-\alpha}, &\text{if $v \ge 1$}\\
	C, &\text{if $v<1$}\\
	\end{array}
\right.,
$
with constant \(C>0\).}
\label{t-pol-type}
\begin{center}
\begin{tabular}{lll}
\hline
\textbf{Transition model}
& \textbf{Radial deformation \(\psi(r)\)}
& \textbf{Tail parameter \(\alpha\)} \\
\hline
Student--\(t\) \((\nu>0)\)
& \(1+\tfrac{1}{\nu}r^2\)
& \(\tfrac{\nu+d_x}{2}\) \\[0.4em]

Cauchy
& \(1+r^2\)
& \(\tfrac{d_x+1}{2}\) \\[0.4em]

Pearson type VII \((\lambda>0)\)
& \(1+\tfrac{1}{\lambda}r^2\)
& \(\alpha>d_x/2\) \\[0.4em]

Generalized Cauchy / rational quadratic
& \(1+r^p,\quad p>0\)
& \(\alpha>d_x/p\) \\[0.4em]
\hline
\end{tabular}
\end{center}

\end{table}


\begin{proposition}\label{prop:epsilon}
Under the assumptions of Theorem~\ref{thm:Dobrushin_Doeblin}, let \(\mX=\mbR^{d_x}\). For any \(s\geq s_\Sigma\), we have the minorization
\beq
\|K\|\leq1-\kappa(s), \quad \text{where} \quad
\kappa(s):=
\int_{\mbR^{d_x}}
\phi\!\bigl(\psi(\|x\|_{\Sigma^{-1}}+s)\bigr)\,\mathrm dx.
\nn
\eeq
The choice \(s=s_\Sigma\) recovers \eqref{eq:kappa}, while \(s=b_\Sigma\) gives \(0<\kappa(b_\Sigma)\leq\kappa(s_\Sigma)\).
For 
\[
S_\Sigma:=\det(\Sigma)^{1/2}\,
\frac{2\pi^{d_x/2}}{\Gamma(d_x/2)},
\]
we obtain the following explicit expressions.

\noindent\textit{Case 1: Exponential radial profiles.}
\(
\phi(v)=C\,\mathrm e^{-v},
\) and \(\psi(r)=ar^\beta\), where \(\mathrm C,a,\beta>0\) yields
\begin{equation}\label{eq:epsilon_Exp}
\kappa(s)
=
S_\Sigma\,C
\sum_{k=0}^{d_x-1}
\binom{d_x-1}{k}
(-s)^{d_x-1-k}
\frac{1}{\beta}
a^{-(k+1)/\beta}
\Gamma\!\left(
\frac{k+1}{\beta},
as^\beta
\right),
\end{equation}
where \(\Gamma(\cdot,\cdot)\) denotes the upper incomplete gamma function.

\noindent\textit{Case 2: Polynomial-tail radial profiles.}
$
\phi(v)=\left\{
	\begin{array}{ll}
	C\,v^{-\alpha}, &\text{if $v \ge 1$}\\
	C, &\text{if $v<1$}\\
	\end{array}
\right.,
$
 and \(\psi(r)=1+ar^2\), where \(a>0\) and \(\alpha>d_x/2\), yields
\begin{equation}\label{eq:epsilon_Pol}
\kappa(s)
=
S_\Sigma\,C
\sum_{k=0}^{d_x-1}
\binom{d_x-1}{k}
(-s)^{d_x-1-k}
I_k(s),
\end{equation}
where 
\beq
I_k(s)
=
\frac{1}{2}
a^{-(k+1)/2}
\left[
B\!\left(
\frac{k+1}{2},
\alpha-\frac{k+1}{2}
\right)
-
B_{z_s}\!\left(
\frac{k+1}{2},
\alpha-\frac{k+1}{2}
\right)
\right],
\label{eq-aux-prop5.7}
\eeq
and 
$z_s
=
\frac{as^2}{1+as^2}$.
The map \(B\) in \eqref{eq-aux-prop5.7} denotes the Beta function, which is finite whenever \(\alpha>d_x/2\), and \(B_{z_s}\) denotes the lower incomplete beta function.
\end{proposition}

\begin{remark}\label{rem:radial_minorization_monotonicity}
The Beta function satisfies
\[
\frac12
a^{-d_x/2}
B\!\Big(
\frac{d_x}{2},
\alpha-\frac{d_x}{2}
\Big)
=
\int_0^\infty
u^{d_x-1}
(1+au^2)^{-\alpha}\,
\mathrm du.
\]
For fixed radial-shape parameters, \(\kappa(s)\) is non-increasing in \(s\). After accounting for the normalization constant \(C\), the conservative choice \(\kappa(b_\Sigma)\) is non-decreasing w.r.t. \(\lambda_{\min}(\Sigma)\) for fixed \(M\), and non-increasing w.r.t. \(M\) for fixed \(\Sigma\). Thus, increasing the smallest eigenvalue of the scale matrix, or decreasing \(M\), can only improve this minorization bound. The exact shift \(s_\Sigma\) depends on the full matrix \(\Sigma\) (not only on \(\lambda_{\min}(\Sigma)\)).
\end{remark}

See Appendix~\ref{ap:Proof_Proposition} for a proof of Proposition \ref{prop:epsilon}.

%
%

%

%

\section{Conclusions}\label{sec6}

This work has developed an operator-theoretic framework for studying stability and error propagation in Bayesian filtering under misspecification of the initial distribution and state dynamics. The observations are r.v.s generated by the true SSM, and the observation mechanism is correctly specified. Within this setting, the mean optimal filtering measures satisfy a linear recursion driven by the predictive operators.

To analyze expected total variation errors, we have derived an exact identity expressing their expected change during an update as a covariance term minus a non-negative remainder. This identity characterizes conditional and averaged non-expansivity conditions. When both
models are available, the covariance and remainder terms can be estimated numerically, e.g., using particle filters (possibly combined with kernel density estimation techniques \cite{Crisan14}) enabling a numerical assessment of these non-expansivity conditions. 
Combining these conditions with the contraction coefficients of the predictive operators has yielded quantitative error bounds that account for potentially expansive updates. The bounds give convergence when accumulated prediction and update contraction dominates the possible expansion introduced by the updates.

We have also characterized stability in expectation through convergence in probability of the filtering error for every alternative initialization. Conditional update non-expansivity makes the error process a non-negative supermartingale. Together with sufficient contraction of the predictive dynamics, this property provides a basis for coalescence a.s. More generally, suitable contraction of the combined prediction and update error bounds also yields convergence a.s.

For misspecified dynamics, we have obtained finite-time estimates that separate the initial error from accumulated dynamical perturbations, measured either uniformly over probability measures or along the optimal filtering sequence. Under averaged update non-expansivity and a uniform joint prediction--update coefficient below one, persistent dynamical misspecification yields an asymptotic error bound proportional to the one-step discrepancy. Under conditional update non-expansivity, summability of the effective forcing and divergent accumulated joint contraction imply almost-sure and \(L^1\) coalescence.

The analytical examples have addressed both components of these conditions. For scalar linear--Gaussian models, an ordering of predictive variances ensures conditional update non-expansivity; this ordering is preserved by the Kalman recursion when both filters use the same dynamics and observation model. Separately, Doeblin minorization yields explicit prediction contraction estimates for truncated kernels and bounded transition functions with elliptically symmetric noise. These results clarify the respective roles of prediction and updating in filtering-error propagation. Extending the analytical verification of the update conditions to broader classes of nonlinear models remains a direction for further work.

\begin{acks}[Acknowledgments]
The authors would like to thank Emre Demirci and Serdar Yuksel for bringing to their attention an error in an earlier version of this paper.
\end{acks}

\begin{funding}
JM and FG acknowledge the support of Spain's \textit{Agencia Estatal de Investigaci\'on} (ref. PID2024-158181NB-I00 NISA), funded by MCIN/AEI/10.13039/501100011033, {\em Comunidad de Madrid} (IDEA-CM project, ref. TEC-2024/COM-89), and {\em Instituto de Salud Carlos III} (ISCIII) through project PRTR-PMP24/00026.

The work of DC was partially supported by the European Research Council (ERC) under the European Union’s Horizon 2020 Research and Innovation Programme (ERC, Grant Agreement No 856408).

\end{funding}


\appendix

%


%

%
\section{Technical proofs of Section \ref{sec2}}

%
\subsection{Proof of Proposition \ref{prop:Ypdf}} \label{proof:Ypdf}

For each fixed \(y_{1:t-1}\in\mathcal Y^{t-1}\), define the function
\[
\psi(y_{1:t-1})
:=
\int_{\mathcal Y}
f(y_{1:t-1},\mathrm{\mathrm{y_t}})\,
p_{\mathrm{y_t}\mid y_{1:t-1}}( \mathrm{\mathrm{y_t}})
\,\mathrm d \mathrm{\mathrm{y_t}} .
\]
We show that \(\psi(Y_{1:t-1})\) is a version of
\(\mathbb E[f(Y_{1:t})\mid \mathcal G_{t-1}]\).
Let \(h:\mathcal Y^{t-1}\to\mathbb R\) be any bounded measurable function. By Fubini--Tonelli and the factorization of the joint density,
\[
p_{Y_{1:t}}(y_{1:t-1},\mathrm{\mathrm{y_t}})
=
p_{\mathrm{y_t}\mid y_{1:t-1}}(\mathrm{\mathrm{y_t}})\,
p_{Y_{1:t-1}}(y_{1:t-1}),
\]
we obtain
\[
\mathbb E\!\left[h(Y_{1:t-1})f(Y_{1:t})\right]
=
\mathbb E\!\left[
h(Y_{1:t-1})\psi(Y_{1:t-1})
\right].
\]
For every \(A\in\mathcal G_{t-1}=\sigma(Y_{1:t-1})\), choose a bounded measurable function \(h:\mathcal Y^{t-1}\to\mathbb R\) such that \(h(Y_{1:t-1})=\textbf{1}_A\). Applying the preceding identity with this choice gives
\[
\int_A f(Y_{1:t})\,\sd\mbP
=\int_A\psi(Y_{1:t-1})\,\sd\mbP,
\qquad\forall A\in\mathcal G_{t-1}.
\]
Therefore,
\[
\mathbb E[f(Y_{1:t})\mid \mathcal G_{t-1}]
=
\psi(Y_{1:t-1})
\quad \text{a.s.}
\]
Finally, by Remark \ref{remark:likelihood_identity},
\[
\psi(Y_{1:t-1})
=
\int_{\mathcal Y}
f(Y_{1:t-1},\mathrm{\mathrm{y_t}})\,
\xi_{Y_{1:t-1}}(g_{\mathrm{\mathrm{y_t}}})
\,\mathrm d \mathrm{\mathrm{y_t}},
\]
which completes the proof.
\qed 

%
\subsection{Proof of Proposition \ref{DUI}} \label{proof:DUI}
Fix \(y_t\in\mathcal Y\), and write \(\lambda:=\mu-\nu\).
The normalizers \(\mu(g_{y_t})\) and \(\nu(g_{y_t})\) are positive
and finite by Remark~\ref{rem:likelihood_assumptions}. We have
\[
\mu(g_{y_t})\bigl(U_{y_t}(\mu)-U_{y_t}(\nu)\bigr)
=g_{y_t}\lambda+
\bigl(\nu(g_{y_t})-\mu(g_{y_t})\bigr)U_{y_t}(\nu).
\]
Taking the total variation mass of both sides, using the triangle
inequality, \(|g_{y_t}\lambda|=g_{y_t}|\lambda|\), and
\(U_{y_t}(\nu)(\mathcal X)=1\), gives
\[
2\mu(g_{y_t})\,\|U_{y_t}(\mu)-U_{y_t}(\nu)\|_{\mathrm{TV}}
\leq |\lambda|(g_{y_t})+
\bigl|\nu(g_{y_t})-\mu(g_{y_t})\bigr|.
\]
The factor \(2\) comes from the total variation convention for
the difference of probability measures. Dividing by
\(2\mu(g_{y_t})\) gives the first bound.
Interchanging \(\mu\) and \(\nu\) gives the same numerator with
denominator \(2\nu(g_{y_t})\). Taking the smaller of these two
bounds yields the claimed denominator
\(2\max\{\mu(g_{y_t}),\nu(g_{y_t})\}\).
\qed

%
\subsection{Proof of Proposition \ref{prop:Link_Op}}\label{proof:Link_Op}

Define the observation kernel
\[
G_t(x,F):=\int_F g_t(y\mid x)\,\mathrm dy,
\quad \text{where} \quad F\in\mathcal B(\mathcal Y).
\]
By the normalisation of the likelihood, $G_t$ is a Markov kernel from $\mathcal X$ to $\mathcal Y$. Its associated linear operator on signed measures is precisely $L_t$. Since this operator preserves probability measures, it maps $\mathcal D(\mathcal X)$ into $\mathcal D(\mathcal Y)$.

Markov kernels are non-expansive in total variation \cite[Lemma~4.3.6, p.~95]{Cappe05}; hence
\[
\|L_t(\lambda)\|_{\mathrm{TV}}
\leq \|\lambda\|_{\mathrm{TV}},
\qquad \lambda\in\mathcal D(\mathcal X).
\]
Therefore, $L_t:\mathcal D(\mathcal X)\to\mathcal D(\mathcal Y)$ is a bounded linear operator with induced norm $\|L_t\|\leq1$.
\qed


%
\subsection{Proof of Proposition \ref{prop:normalizing_constant_positive}}\label{proof:normalizing_constant_positive}

Fix \(t\geq1\), \(y_t\in\mathcal Y\) and
\(\rho\in\mathcal P(\mathcal X)\). Since \(g_{y_t}(x)>0\) for
every \(x\in\mathcal X\) and \(\rho(\mathcal X)=1\), we have
\(\rho(g_{y_t})>0\). Boundedness also gives
\[
\rho(g_{y_t})\leq\|g_{y_t}\|_\infty\rho(\mathcal X)
=\|g_{y_t}\|_\infty<\infty.
\]
Applying these observations with \(\rho=\mu\) and \(\rho=\nu\)
proves the claim.
\qed

%

\section{Technical proofs of Section \ref{sec3}}

\subsection{Proof of Theorem \ref{lemma:ExUpd}}\label{proof:ExUTV}
The random measures \(\mu,\nu\) can be represented as measurable kernels indexed by \(y_{1:t-1}\). Joint measurability of \(g_t\) and the positive, finite normalizers imply that their updates are measurable kernels indexed by \((y_{1:t-1},\mathrm y_t)\). By the measurable Radon--Nikodym theorem \cite[Theorem~2.3]{novikov2005hahn}, each pair admits jointly measurable densities w.r.t. its sum. Taking the set where the first density is at least the second gives the required positive Hahn set. The resulting set evaluations and remainder integrands are therefore measurable.

Since $A_t$ is the positive Hahn set of
\(
U_{Y_t}(\mu)-U_{Y_t}(\nu),
\)
we have
\(
\left\|
U_{Y_t}(\mu)-U_{Y_t}(\nu)
\right\|_{\mathrm{TV}}
=
U_{Y_t}(\mu)(A_t)-U_{Y_t}(\nu)(A_t).
\)
Moreover,
\[
A_t
=
(A_t\cap\widetilde A_t)
\dot\cup
(A_t\setminus\widetilde A_t) \quad \text{and} \quad \widetilde A_t
=
(A_t\cap\widetilde A_t)
\dot\cup
(\widetilde A_t\setminus A_t),
\]
where $\dot\cup$ denotes the disjoint union of sets. Hence
\[
U_{Y_t}(\mu)(A_t)
=
U_{Y_t}(\mu)(\widetilde A_t)
+
U_{Y_t}(\mu)(A_t\setminus\widetilde A_t)
-
U_{Y_t}(\mu)(\widetilde A_t\setminus A_t).
\]
Therefore,
\begin{equation}
\label{eq:conditional-tv-decomposition}
\begin{aligned}
&\mathbb E^{L_t(\mu)}
\left[
\left\|
U_{Y_t}(\mu)-U_{Y_t}(\nu)
\right\|_{\mathrm{TV}}
\right]
=
\mathbb E^{L_t(\mu)}
\left[
U_{Y_t}(\mu)(\widetilde A_t)
\right]
-
\mathbb E^{L_t(\mu)}
\left[
U_{Y_t}(\nu)(A_t)
\right]
\\
&\quad+
\mathbb E^{L_t(\mu)}
\left[
U_{Y_t}(\mu)(A_t\setminus\widetilde A_t)
\right]
-
\mathbb E^{L_t(\mu)}
\left[
U_{Y_t}(\mu)(\widetilde A_t\setminus A_t)
\right].
\end{aligned}
\end{equation}

Since $\mu$ and $\widetilde A_t$ are
$\mathcal G_{t-1}$-measurable (and therefore do not depend on $y_t$),
\[
\begin{aligned}
\mathbb E^{L_t(\mu)}
\left[
U_{Y_t}(\mu)(\widetilde A_t)
\right]
&=
\int_{\mY}
U_{y_t}(\mu)(\widetilde A_t)
\mu(g_{y_t})\,dy_t
=
\int_{\mY}
\int_{\widetilde A_t}
g_{y_t}(x)\,\mu(dx)\,dy_t
\\
&=
\int_{\widetilde A_t}
\left(
\int_{\mY}
g_{y_t}(x)\,dy_t
\right)\mu(dx)
=
\mu(\widetilde A_t),
\end{aligned}
\]
where we have used Tonelli's theorem and the fact that $g_{y_t}(x)$ is a pdf i.e. \eqref{eq:g_normalized}.
Adding and subtracting $\nu(\widetilde A_t)$ in
\eqref{eq:conditional-tv-decomposition}, and using that
$\widetilde A_t$ is the positive Hahn set of $\mu-\nu$, we obtain
\[
\mu(\widetilde A_t)-\nu(\widetilde A_t)
=
\|\mu-\nu\|_{\mathrm{TV}}.
\]
Moreover, note that
\(
\nu(\widetilde A_t)
=
\mathbb E^{L_t(\nu)}
\left[
U_{Y_t}(\nu)(\widetilde A_t)
\right].
\)
Consequently,
\begin{equation}
\label{eq:conditional-before-change-measure}
\begin{aligned}
\mathbb E^{L_t(\mu)}
\left[
\left\|
U_{Y_t}(\mu)-U_{Y_t}(\nu)
\right\|_{\mathrm{TV}}
\right]
&=
\|\mu-\nu\|_{\mathrm{TV}}
+
\mathbb E^{L_t(\nu)}
\left[
U_{Y_t}(\nu)(\widetilde A_t)
\right]
-
\mathbb E^{L_t(\mu)}
\left[
U_{Y_t}(\nu)(A_t)
\right]
\\
&+
\mathbb E^{L_t(\mu)}
\left[
U_{Y_t}(\mu)(A_t\setminus\widetilde A_t)
\right]
-
\mathbb E^{L_t(\mu)}
\left[
U_{Y_t}(\mu)(\widetilde A_t\setminus A_t)
\right].
\end{aligned}
\end{equation}
Since $g_{y_t}>0$, the conditional observation measures
$L_t(\mu)$ and $L_t(\nu)$ are equivalent, with Radon--Nikodym derivatives
\[
\frac{dL_t(\mu)}{dL_t(\nu)}(Y_t)
=
\frac{\mu(g_{Y_t})}{\nu(g_{Y_t})}
\quad \text{and} \quad
\frac{dL_t(\nu)}{dL_t(\mu)}(Y_t)
=
\frac{\nu(g_{Y_t})}{\mu(g_{Y_t})}.
\]
Hence, by the Radon--Nikodym change-of-measure formula, for any
conditionally integrable random variable $F_t$ which, conditionally on
$\mathcal G_{t-1}$, is a measurable function of $Y_t$,
\[
\mathbb E^{L_t(\mu)}[F_t]
=
\mathbb E^{L_t(\nu)}
\left[
\frac{\mu(g_{Y_t})}{\nu(g_{Y_t})}F_t
\right], \qquad \mathbb E^{L_t(\nu)}[F_t]
=
\mathbb E^{L_t(\mu)}
\left[
\frac{\nu(g_{Y_t})}{\mu(g_{Y_t})}F_t
\right].
\]
Now, for any measurable set $C\subseteq\mX$, we can write
\begin{equation}
\label{eq:conditional-update-change-measure}
\begin{aligned}
\frac{\mu(g_{Y_t})}{\nu(g_{Y_t})}
U_{Y_t}(\mu)(C)
&=
\frac{\mu(g_{Y_t}\mathbf 1_C)}
{\nu(g_{Y_t})}
=
U_{Y_t}(\nu)(C)
+
\frac{
(\mu-\nu)(g_{Y_t}\mathbf 1_C)
}{
\nu(g_{Y_t})
}.
\end{aligned}
\end{equation}
Since
\(
A_t\setminus\widetilde A_t
\subseteq
\widetilde B_t,
\)
and $\widetilde B_t$ is the negative Hahn set of $\mu-\nu$,
\[
(\mu-\nu)
\big(
g_{Y_t}\mathbf 1_{A_t\setminus\widetilde A_t}
\big)
=
-
|\mu-\nu|
\big(
g_{Y_t}\mathbf 1_{A_t\setminus\widetilde A_t}
\big).
\]
Thus
\[
\begin{aligned}
&\frac{\mu(g_{Y_t})}{\nu(g_{Y_t})}
U_{Y_t}(\mu)(A_t\setminus\widetilde A_t)
=
U_{Y_t}(\nu)(A_t\setminus\widetilde A_t)
-
\frac{
|\mu-\nu|
\big(
g_{Y_t}\mathbf 1_{A_t\setminus\widetilde A_t}
\big)
}{
\nu(g_{Y_t})
}.
\end{aligned}
\]
Similarly,
\(
\widetilde A_t\setminus A_t
\subseteq
\widetilde A_t,
\)
where $\mu-\nu$ is positive, and therefore
\[
\begin{aligned}
&\frac{\mu(g_{Y_t})}{\nu(g_{Y_t})}
U_{Y_t}(\mu)(\widetilde A_t\setminus A_t)
=
U_{Y_t}(\nu)(\widetilde A_t\setminus A_t)
+
\frac{
|\mu-\nu|
\big(
g_{Y_t}\mathbf 1_{\widetilde A_t\setminus A_t}
\big)
}{
\nu(g_{Y_t})
}.
\end{aligned}
\]
Using the identities above \eqref{eq:conditional-before-change-measure}, and observing that
\(
(A_t\setminus\widetilde A_t)
\dot\cup
(\widetilde A_t\setminus A_t)
=
A_t\triangle\widetilde A_t,
\)
we obtain
\begin{eqnarray}
\mathbb E^{L_t(\mu)}
\left[
\left\|
U_{Y_t}(\mu)-U_{Y_t}(\nu)
\right\|_{\mathrm{TV}}
\right] &=&
\|\mu-\nu\|_{\mathrm{TV}}
+
\mathbb E^{L_t(\nu)}
\left[
U_{Y_t}(\nu)(A_t)
\right]
\nonumber
\\
&& - \mathbb E^{L_t(\mu)}
\left[
U_{Y_t}(\nu)(A_t)
\right]
- \mathbb E^{L_t(\nu)}
\left[
\frac{
|\mu-\nu|
\big(
g_{Y_t}\mathbf 1_{A_t\triangle\widetilde A_t}
\big)
}{
\nu(g_{Y_t})
}
\right].
\label{eq:conditional-pre-covariance}
\end{eqnarray}

We now identify the covariance term. Set
\(
Z_t
:=
\frac{\nu(g_{Y_t})}{\mu(g_{Y_t})}
=
\frac{dL_t(\nu)}{dL_t(\mu)}(Y_t)
\)
and note that
\[
\mathbb E^{L_t(\mu)}[Z_t]
=\int_{\mathcal Y}\nu(g_{\mathrm y_t})\,\mathrm d\mathrm y_t=1.
\]
Furthermore,
\(
0\leq U_{Y_t}(\nu)(A_t)\leq1,
\)
and hence
\[
\begin{aligned}
\mathbb E^{L_t(\mu)}
\left[
Z_tU_{Y_t}(\nu)(A_t)
\right]
&\leq
\mathbb E^{L_t(\mu)}[Z_t]=1.
\end{aligned}
\]
Therefore,
\(
Z_tU_{Y_t}(\nu)(A_t)
\)
is integrable, so the covariance
\[
\operatorname{Cov}^{L_t(\mu)}
\left(
Z_t,
U_{Y_t}(\nu)(A_t)
\right)=\mbE^{L_t(\mu)}[Z_t U_{Y_t}(\nu)(A_t)] -\mbE^{L_t(\mu)}[U_{Y_t}(\nu)(A_t)] \mbE^{L_t(\mu)}[Z_t]
\]
is well defined.
By the change-of-measure identity,
\[
\mathbb E^{L_t(\mu)}
\left[
Z_tU_{Y_t}(\nu)(A_t)
\right]
=
\mathbb E^{L_t(\nu)}
\left[
U_{Y_t}(\nu)(A_t)
\right].
\]
Since $\mathbb E^{L_t(\mu)}[Z_t]=1$, it follows that
\begin{equation}
\label{eq:conditional-covariance}
\begin{aligned}
&\operatorname{Cov}^{L_t(\mu)}
\left(
\frac{\nu(g_{Y_t})}{\mu(g_{Y_t})},
U_{Y_t}(\nu)(A_t)
\right)
=
\mathbb E^{L_t(\nu)}
\left[
U_{Y_t}(\nu)(A_t)
\right]
-
\mathbb E^{L_t(\mu)}
\left[
U_{Y_t}(\nu)(A_t)
\right].
\end{aligned}
\end{equation}
Substituting \eqref{eq:conditional-covariance} into \eqref{eq:conditional-pre-covariance}, and using the reverse change-of-measure identity
\[
\begin{aligned}
&\mathbb E^{L_t(\nu)}
\left[
\frac{
|\mu-\nu|
\big(
g_{Y_t}\mathbf 1_{A_t\triangle\widetilde A_t}
\big)
}{
\nu(g_{Y_t})
}
\right]
=
\mathbb E^{L_t(\mu)}
\left[
\frac{
|\mu-\nu|
\big(
g_{Y_t}\mathbf 1_{A_t\triangle\widetilde A_t}
\big)
}{
\mu(g_{Y_t})
}
\right],
\end{aligned}
\]
we arrive at \eqref{eq:conditional-update-identity}.

Finally, the remainder expectation in \eqref{eq:conditional-update-identity} is finite. Indeed, since \(|\mu-\nu|\leq\mu+\nu\) and \(\mathbf 1_{A_t\triangle\widetilde A_t}\leq1\),
\[
0\leq
\mathbb E^{L_t(\mu)}\!\left[
\frac{|\mu-\nu|\bigl(g_{Y_t}\mathbf 1_{A_t\triangle\widetilde A_t}\bigr)}
{\mu(g_{Y_t})}
\right]
\leq 1+\mathbb E^{L_t(\mu)}[Z_t]=2.
\]
The same bound holds for the \(L_t(\nu)\)-expectation with denominator \(\nu(g_{Y_t})\). Since the bound \(2\) holds for every observation history \(y_{1:t-1}\), averaging over the observation-history distribution yields
\[
0\leq
\mathbb E\!\left[
\mathbb E^{L_t(\mu)}\!\left[
\frac{|\mu-\nu|\bigl(g_{Y_t}\mathbf 1_{A_t\triangle\widetilde A_t}\bigr)}
{\mu(g_{Y_t})}
\right]
\right]\leq2.
\]
When \(\mu=\xi_{Y_{1:t-1}}\), the inner expectation is conditional expectation given \(\mathcal G_{t-1}\); hence the tower property gives integrability of the remainder integrand under the true model.
\qed

\subsection{Proof of Lemma~\ref{lem:Bounded_Conv_Equiv}}\label{proof:Bounded_Conv_Equiv}

Let us first recall, for completeness, the definition of uniform integrability and the classical Vitali's convergence theorem. 

\begin{definition}[Uniform integrability]
A class \(\mathcal C\) of r.v.s is called
uniformly integrable (\(\mathrm{UI}\)) if given  \(\epsilon>0\) there exists
\(R\in[0,\infty)\) such that
\[
\mbE[
|Z|\,\mathbf 1_{\{|Z|>R\}}
]=\int_{|Z|>R} |Z| \, \sd \mbP
<\epsilon, \quad \forall Z \in \mC,
\]
where \(\mathbf 1\) denotes the indicator function.
\end{definition}
\begin{remark}\label{rem:UB_UI}
Note that any almost sure uniformly bounded (\(\mathrm{UB}\)) class of random variables is \(\mathrm{UI}\). Indeed, if there exists $L>0$ such that 
\(
|Z|\leq L,
\) a.s. 
 \(\forall Z\in \mC\), then
\(
\mathbb E[|Z|\mathbf 1_{\{|Z|>R\}}]=0,
\)
for every \(R>L\). For brevity, we write
\(
\mathrm{UB}
\Rightarrow
\mathrm{UI}.
\) A weaker sufficient condition is uniform boundedness in \(L^p\), namely if there
exist \(p>1\) and \(L>0\) such that
\(
\sup_{Z\in\mathcal C}\mathbb E[|Z|^p]\leq L,
\)
then \(\mathcal C\) is uniformly integrable; see, for instance, Example 1.3.35
\cite{aggoun2004measure}.  This implication fails for \(p=1\); see the first example in Chapter~13 of \cite{Williams91}.
\end{remark}

\begin{theorem}\label{thm:Williams_L1}
Let \(\{Z_n\}_{n\geq1}\subset L^1(\mathbb P)\) and let \(Z\in L^1(\mathbb P)\). Then
\(
Z_n\xrightarrow{L^1}Z
\)
if and only if the following conditions hold:
\begin{enumerate}
\item
\(
Z_n\xrightarrow{\mathbb P}Z;
\)
\item
\(
\{Z_n\}_{n\geq1}
\)
is \(\mathrm{UI}\).
\end{enumerate}
\end{theorem}

This is Vitali's convergence theorem; a proof can be found in \cite[Theorem~13.7]{Williams91}.

Let us now proceed with the proof of Lemma~\ref{lem:Bounded_Conv_Equiv}. We first establish that, for the sequence \(\{Z_t\}_{t\geq1}\), convergence in probability and convergence in \(L^1\) coincide.

Since
\(
0\leq Z_t\leq1,
\) a.s. for every 
\(t\geq1\), the family \(\{Z_t\}_{t\geq1}\) is \(\mathrm{UI}\) by Remark~\ref{rem:UB_UI}. Hence, for any \(X_\infty\in
L^1(\mathbb P)\), Vitali's convergence theorem yields
\begin{equation}\label{eq:Bounded_L1_P}
Z_t\xrightarrow{L^1}X_\infty
\quad\Longleftrightarrow\quad
Z_t\xrightarrow{\mathbb P}X_\infty.
\end{equation}

To prove \(i)\Rightarrow ii)\), assume first that
\(
Z_t\xrightarrow{\mathbb P}0.
\)
By \eqref{eq:Bounded_L1_P},
\(
Z_t\xrightarrow{L^{p_0}}0,
\)
with \(p_0=1\).

To prove \(ii)\Rightarrow iii)\), assume that, for some \(p_0\geq1\),
\(
Z_t\xrightarrow{L^{p_0}}0.
\)
Let \(p\geq1\). If \(p>p_0\), since \(0\leq Z_t\leq1\), then 
\(
Z_t^p\leq Z_t^{p_0},
\)
and hence
\(
\mathbb E[Z_t^p]
\leq
\mathbb E[Z_t^{p_0}]
\longrightarrow0.
\)
If \(p<p_0\), then the map
\(x\mapsto x^{p_0/p}\) is convex. Hence, for any probability measure
\(\mu\), Jensen's inequality implies that, for every integrable function
\(f\),
\[
\left(
\int |f|^p\,\mathrm d\mu
\right)^{p_0/p}
\leq
\int |f|^{p_0}\,\mathrm d\mu.
\]
Consequently,
\(
\|Z_t\|_{L^p}
\leq
\|Z_t\|_{L^{p_0}}
\longrightarrow0.
\)
Therefore,
\(
Z_t\xrightarrow{L^p}0,
\;\forall p\geq1.
\)

\noindent \(iii)\Rightarrow iv)\)
Assume that
\(
Z_t\xrightarrow{L^p}0, \;\forall p\ge1.
\) In particular \(
Z_t\xrightarrow{L^1}0,
\) 
 it follows from \eqref{eq:Bounded_L1_P} that
\(
Z_t\xrightarrow{\mathbb P}0.
\)
Since \(0\leq Z_t^p\leq Z_t,\) a.s. for every $t\ge1$. Then  
\beq
\label{ineq:E_t^pleqE_t} 
\mbP(Z_t^p\geq \delta) \leq \mbP(Z_t\geq \delta), \quad \forall p\ge1.
\eeq 
Hence 
\( Z_t^p\xrightarrow{\mathbb P}0, \; \forall p\ge1.
    \)

Since \(iv)\Rightarrow i)\) trivially, the four statements are equivalent.

Now, under any of these equivalent conditions, in particular,
\(
Z_t\xrightarrow{L^2}0.
\)
We have
\[
\operatorname{Var}(Z_t)=\mbE[Z_t^2]-\mbE[Z_t]^2
\leq
\mathbb E[Z_t^2]
=\|Z_t\|_{L^2(\mbP)}^2
\longrightarrow0.
\]
Finally, if
\(
\sum_{t=0}^{\infty}\mathbb P(Z_t\geq \delta)
<
\infty, \; \forall \delta>0,
\)
 the Borel--Cantelli lemma yields
\(
\mathbb P(Z_t< \delta \ \mathrm{ev})=1.
\)
In other words, with probability one, there exists \(N\in\mathbb N\) such that
\(
Z_t<\delta,
\;
\forall\, t\geq N.
\)
Take \(\delta_q\in\mathbb Q^+\) and define
\[
A_{\delta_q}
:=
\left\{
\exists\, N\in\mathbb N
\text{ such that }
Z_t\leq \delta_q,
\quad
\forall\, t\geq N
\right\}.
\]
From the previous argument, we have proved that
\(
\mathbb P(A_{\delta_q})=1,
\;
\forall\,\delta_q\in\mathbb Q^+.
\)
Now define the set
\(
A
:=
\bigcap_{\delta_q\in\mathbb Q^+} A_{\delta_q}.
\)
Since \(\mathbb Q^+\) is countable and
\(
\mathbb P(A_{\delta_q})=1,
\)
 \( \forall \delta_q\in\mathbb Q^+\), it follows from the continuity of probability measures from above (Lemma~1.10 in \cite{Williams91}) that 
\(
\mathbb P(A)=1.
\)
We now show that \(Z_t(\omega)\to0\) for every \(\omega\in A\).

Let
\(\omega\in A\) and let \(\varepsilon>0\). Since \(\mathbb Q^+\) is dense in
\(\mathbb R^+\), there exists \(\delta_q\in\mathbb Q^+\) such that
\(
0<\delta_q<\varepsilon.
\)
Since \(\omega\in A\subseteq A_{\delta_q}\), there exists
\(N=N(\omega,\delta_q)\in\mathbb N\) such that
\(
Z_t(\omega)\leq \delta_q< \varepsilon,
\;
\forall\,t\geq N.
\)
Since \(\varepsilon>0\) was arbitrary, we conclude that
\(
Z_t(\omega)\to0.
\)
Hence \(Z_t\to0\) on a set \(A\) with \(\mathbb P(A)=1\), and therefore
\(
Z_t\xrightarrow{\mathrm{a.s.}}0.
\) By \eqref{ineq:E_t^pleqE_t} and the comparison test (Theorem 8.20 in~\cite{apostol1958mathematical})
\[\sum_{t=0}^{\infty}\mathbb P(Z_t^p\geq\delta)
\leq \sum_{t=0}^{\infty}\mathbb P(Z_t\geq\delta)<\infty.
\] 
Hence \(Z_t^p \xrightarrow{\mathrm{a.s.}}0\), by the same arguments used for $Z_t$.
\qed

\section{Technical result for section \ref{sec4} }\label{apRobbins-Siegmund}

\begin{theorem}[Robbins--Siegmund \cite{robbins1971convergence}]\label{thm:robbins_siegmund}
Let \(\{X_t\}_{t\geq 0}\), \(\{A_t\}_{t\geq 0}\), \(\{B_t\}_{t\geq 0}\), and
\(\{C_t\}_{t\geq 0}\) be sequences of nonnegative integrable r.v.s
on a probability space, adapted to a filtration \(\{\mathcal F_t\}_{t\geq 0}\).
Assume that
\(
\sum_{i=0}^{\infty} A_i < \infty,
\;
\sum_{i=0}^{\infty} C_i < \infty
\; \text{a.s.}
\)
and that, for every \(t\in\mathbb N\),
\[
\mathbb E\!\left[X_{t+1}\mid \mathcal F_t\right]
\leq
(1+A_t)X_t - B_t + C_t
\quad \text{a.s.}
\]
Then \(\{X_t\}_{t\geq 0}\) converges almost surely and
\(
\sum_{i=0}^{\infty} B_i < \infty
\; \text{a.s.}
\)
\end{theorem}

%
\section{Technical proofs of Section \ref{sec5}}

%
\subsection{Proof of Proposition~\ref{prop:scalar-gaussian-nonexpansion}}
\label{Ap_P5.1}

Fix \(t\geq1\) and condition on the past observations.
We suppress the time index and the dependence on the
observation history throughout the proof.

If \(H=0\), the likelihood is independent of the state,
so \(U_y(\xi)=\xi\), \(U_y(\nu)=\nu\), and \(r_y=1\).
The covariance therefore vanishes (since $r_y$ is a constant), and the update preserves
the total variation distance.
If \(\xi=\nu\), the conclusion is also immediate.
Hence, we assume \(H\neq0\) and \(\xi\neq\nu\) hereafter.

Let \(\Phi\) denote the standard normal distribution function
and \(\varphi=\Phi'\) its density.
The updated measures are
\[
U_y(\xi)=\mathcal N(m_y^+,P^+),
\qquad
U_y(\nu)=\mathcal N(\widehat m_y^+,\widehat P^+),
\]
where
\beq
m_y^+ =\frac{Rm+PHy}{S}, \quad
\widehat m_y^+ =\frac{R\widehat m+\widehat P Hy}{\widehat S}, \quad
P^+ = \frac{PR}{S}, \quad
\widehat P^+
=\frac{\widehat P R}{\widehat S}
\label{eq-aux-9.9}
\eeq
and
\beq
S=H^2P+R,
\quad
\widehat S=H^2\widehat P+R.
\label{eq-aux-9.95}
\eeq
Let \(f_y\) and \(\widehat f_y\) denote the densities of $U_y(\xi)$ and $U_y(\nu)$, respectively, and take
\[
A_y:=\{x:f_y(x)\geq\widehat f_y(x)\}.
\]
We split the proof in two cases. If $\widehat P = P$, the analysis is relatively straightforward, while the case $\widehat P >  P$ requires more effort.

%
\subsubsection*{Case 1: \(\widehat P=P\)}

If $\widehat P=P$ and $\xi \ne \nu$, then $\widehat P^+=P^+$ and $
\widehat m_y^+-m_y^+
=
\frac{R}{S}(\widehat m-m)
$. Moreover,
\beq
\log\frac{f_y(x)}{\widehat f_y(x)}
=
\frac{(x-\widehat m_y^+)^2-(x-m_y^+)^2}{2P^+}
=
\frac{\widehat m_y^+-m_y^+}{P^+}
\left(\frac{m_y^++\widehat m_y^+}{2}-x\right).
\nn
\eeq
The densities are therefore equal at exactly one point,
their midpoint
\[
c_y:=\frac{m_y^++\widehat m_y^+}{2},
\]
and it follows that
\[
A_y=
\begin{cases}
(-\infty,c_y], & \text{if}~~\widehat m_y^+>m_y^+,\\
[c_y,\infty),  & \text{if}~~\widehat m_y^+<m_y^+.
\end{cases}
\]
When \(\widehat m_y^+>m_y^+\), the computation of $q_y = U_y(\nu)(A_y)$ yields
\beq
q_y
=
\Phi\left(
  \frac{c_y-\widehat m_y^+}{\sqrt{P^+}}
\right)
=
\Phi\left(
  -\frac{\widehat m_y^+-m_y^+}{2\sqrt{P^+}}
\right).
\label{eq-aux-10}
\eeq
When \(\widehat m_y^+<m_y^+\), the identity \(1-\Phi(u)=\Phi(-u)\) yields the same expression as in \eqref{eq-aux-10}, with the mean difference replaced by its absolute value. Thus,
\[
q_y
=
\Phi\left(
  -\frac{|\widehat m_y^+-m_y^+|}{2\sqrt{P^+}}
\right)
=
\Phi\left(
  -\frac{|\widehat m-m|}{2\sqrt P}\sqrt{\frac RS}
\right),
\]
where we have used \eqref{eq-aux-9.9} for the second equality. This quantity is independent of \(y\), so $\operatorname{Cov}^{L(\xi)}(r_Y,q_Y)=0$, which completes the proof for $\widehat P=P$.

%
\subsubsection*{Case 2: \(\widehat P>P\)}

It is apparent that the ordering of the updated variances is preserved,
\[
\widehat P^+-P^+
=
\frac{R^2(\widehat P-P)}{S\widehat S}>0,
\]
and the difference $\widehat P^+-P^+$ does not depend on \(y\). The difference between the updated means can be written as
\[
d_y:=\widehat m_y^+-m_y^+
=
\frac{R}{S\widehat S}
\left[
  \widehat m S-m\widehat S
  +H(\widehat P-P)y
\right]
\]
and, since \(H\neq0\), this affine function vanishes at a unique point in observation space, namely,
\[
y_\star
=
\frac{m\widehat S-\widehat m S}{H(\widehat P-P)}.
\]
As a consequence, we may re-write 
\[
d_y
=
\frac{HR(\widehat P-P)}{S\widehat S}(y-y_\star).
\]

In the sequel, we
\begin{itemize}
\item characterize the density ratio $r_y$,
\item identify the positive Hahn set $A_y$ and its probability $q_y$,
\item prove monotonicity properties for $r_y$ and $q_y$ and
\item use these properties to show that ${\rm Cov}^{L(\xi)}(r_Y,q_Y) \le 0$.
\end{itemize}

%
\subsubsection*{The observation-density ratio $r_y$}

The Gaussian observation model gives
\[
\log r_y
=
\frac12\log\frac{S}{\widehat S}
+\frac{(y-Hm)^2}{2S}
-\frac{(y-H\widehat m)^2}{2\widehat S}
\]
and differentiating w.r.t. $y$ identifies the difference of updated means,
\beq
\frac{\mathrm d}{\mathrm dy}\log r_y
=
\frac{y-Hm}{S}
-\frac{y-H\widehat m}{\widehat S}
=
\frac{H}{R}d_y
=
\frac{H^2(\widehat P-P)}{S\widehat S}(y-y_\star).
\nn
\eeq
Integrating from \(y_\star\) to \(y\), we readily obtain
\beq
r_y
=
C\exp\left(
  \frac{H^2(\widehat P-P)}{2S\widehat S}
  (y-y_\star)^2
\right),
\quad \text{with} \quad
C:=r_{y_\star}>0.
\label{eq-aux-12}
\eeq

%
\subsubsection*{Probability of the positive Hahn set $A_y$}

To compute \(q_y=U_y(\nu)(A_y)\), let us write
\[
k:=\frac{\widehat P^+}{P^+}
=\frac{\widehat P S}{P\widehat S}>1
\]
and recall that $A_y=\{x:f_y(x)\geq\widehat f_y(x)\}$ where $f_y$ and $\widehat f_y$ are Gaussian densities with variances \(P^+\) and \(\widehat P^+=kP^+\), respectively. The pdf ratio can be easily shown to be
\[
\frac{f_y(x)}{\widehat f_y(x)}
=
\sqrt{k}\,
\exp\left[
-\frac12
\left(
\frac{(x-m_y^+)^2}{P^+}
-
\frac{(x-\widehat m_y^+)^2}{kP^+}
\right)
\right]
\]
and straightforward manipulations show that
\begin{equation}\label{eq:x_in_A_y}
x\in A_y
\quad\Longleftrightarrow\quad
\log \frac{ f_y(x) }{ \widehat f_y(x) } \ge 0
\quad\Longleftrightarrow\quad
\frac{(x-m_y^+)^2}{P^+}
-
\frac{(x-\widehat m_y^+)^2}{kP^+}
\leq\log k.
\end{equation}

Under \(U_y(\nu)=\mathcal N(\widehat m_y^+,kP^+)\), let us introduce the r.v.
\[
Z:=\frac{X-\widehat m_y^+}{\sqrt{kP^+}}
\sim\mathcal N(0,1).
\]
Using the right-hand-side inequality in \eqref{eq:x_in_A_y}, we readily see that
\[
X \in A_y \quad\Longleftrightarrow\quad
(k-1)Z^2
+2\sqrt{k}\,\frac{d_y}{\sqrt{P^+}}Z
+\frac{d_y^2}{P^+}
\leq\log k.
\]
Dividing by \(k-1\) and completing the square yields
\beq
X \in A_y \quad\Longleftrightarrow\quad 
\left(
Z+\frac{\sqrt{k}\,d_y}{\sqrt{P^+}(k-1)}
\right)^2
\leq
\frac{d_y^2}{P^+(k-1)^2}
+\frac{\log k}{k-1}.
\label{eq-aux-11}
\eeq

For conciseness, let us denote
\beq
\tau_y:=\frac{d_y^2}{P^+(k-1)^2}
\quad\text{and}\quad
c_k:=\frac{\log k}{k-1}
\label{eq-def-t-c}
\eeq
for the two quantities on the right hand side of \eqref{eq-aux-11}. Here \(\tau_y\) is a rescaled squared separation between the posterior means. The inequality in \eqref{eq-aux-11} implies, by symmetry of \(Z\), that the probability of the event $X \in A_y$ remains the same when \(d_y\) changes sign. Therefore we can write \(q_y=q(\tau_y)\), where, for any \(\tau\geq0\),
\beq
q(\tau)
:=
\Phi\left(\sqrt{\tau+c_k}-\sqrt{k\tau}\right)
+
\Phi\left(\sqrt{\tau+c_k}+\sqrt{k\tau}\right)
-1.
\label{eq-def-q-t}
\eeq

%
\subsubsection*{Monotonicity of $r$ and $q$}

Using \eqref{eq-aux-9.9}, \eqref{eq-aux-9.95} and \eqref{eq-def-t-c}, we can write
\[
k-1=\frac{R(\widehat P-P)}{P\widehat S}
\quad\text{and}\quad
\tau_y=\frac{PH^2}{RS}(y-y_\star)^2.
\]
The ratio $r_y$ in \eqref{eq-aux-12} can therefore be expressed as
\[
r_y=r(\tau_y)
\quad\text{and}\quad
r(\tau)=C\exp\left(\frac{k-1}{2}\tau\right),
\]
hence the function $r$ is monotonically increasing with $\tau_y$.

We now look into $q_y$. For \(\tau>0\), let us write
\beq
\rho:=\sqrt{\tau+c_k},
\quad
a:=\sqrt{k\tau}
\quad\text{and}\quad
q(\tau)=\Phi(\rho-a)+\Phi(\rho+a)-1.
\label{eq-aux-14}
\eeq
Differentiating $q(\tau)$ in \eqref{eq-def-q-t} w.r.t. $\tau$ and setting \(z:=\rho a\), we obtain
\beqa
q'(\tau)
&=&
\frac{\varphi(\rho-a)+\varphi(\rho+a)}{2\rho}
-
\frac{a}{2\tau}
\bigl(\varphi(\rho-a)-\varphi(\rho+a)\bigr) \nn\\
&=&
\frac{\varphi(\rho-a)+\varphi(\rho+a)}{2\rho}
\left(1-\frac{z\tanh z}{\tau}\right),
\label{eq-aux-13}
\eeqa
where the second equality follows from
\[
\frac{\varphi(\rho-a)-\varphi(\rho+a)}
     {\varphi(\rho-a)+\varphi(\rho+a)}
=
\tanh(\rho a).
\]
The first factor on the right hand side of \eqref{eq-aux-13} is positive, so in order to show that $q(\tau)$ is monotonically decreasing w.r.t. $\tau>0$ it suffices to prove that \(z\tanh z>\tau\).

The inequality
\[
\log k
=
\int_1^k\frac{\mathrm du}{u}
>
\int_1^k\frac1k\,\mathrm du
=
\frac{k-1}{k}
\]
combined with \eqref{eq-def-t-c} yields
\[
kc_k=\frac{k\log k}{k-1}>1.
\]
Recalling that $z=\rho a$, with $\rho$ and $a$ as defined in \eqref{eq-aux-14}, the inequality above yields
\[
z^2=k\tau^2+kc_k\tau>\tau(\tau+1).
\]
Also, \(\sinh z>z\) for \(z>0\), which implies
\[
\tanh z
=
\frac{\sinh z}{\sqrt{1+\sinh^2 z}}
>
\frac{z}{\sqrt{1+z^2}}.
\]
As consequence,
\[
z\tanh z
>
\frac{z^2}{\sqrt{1+z^2}}
>
\frac{\tau(\tau+1)}{\sqrt{1+\tau(\tau+1)}}
>
\tau.
\]
The middle inequality uses the monotonicity of
\(u\mapsto u/\sqrt{1+u}\); the last follows from
\(1+\tau(\tau+1)<(\tau+1)^2\).
Thus \(q'(\tau)<0\) for \(\tau>0\).
By continuity at zero, \(q\) is non-increasing on
\([0,\infty)\).

%
\subsubsection*{Sign of the covariance}

Let \(\tau_Y'\) be an independent copy of \(\tau_Y\), with \(Y\sim L(\xi)\). Independence and equality in distribution give
\[
\begin{aligned}
\operatorname{Cov}^{L(\xi)}(r_Y,q_Y)
&=
\frac12\,
\mathbb E\left[
  \bigl(r(\tau_Y)-r(\tau_Y')\bigr)
  \bigl(q(\tau_Y)-q(\tau_Y')\bigr)
\right]\leq0,
\end{aligned}
\]
because \(r(\tau_Y)\) is increasing and \(q(\tau_Y)\) is non-increasing. The expectations are finite since
\(\mathbb E^{L(\xi)}[r_Y]=1\) and \(0\leq q_Y\leq1\).

Thus the covariance is non-positive in both cases.
Together with the non-negative remainder
\(\mT_{\rm rem}\) in Theorem~\ref{lemma:ExUpd}, this yields
\[
\mathbb E^{L(\xi)}
\left[
  \|U_Y(\xi)-U_Y(\nu)\|_{\mathrm{TV}}
\right]
\leq
\|\xi-\nu\|_{\mathrm{TV}}.
\]
Since the argument applies for almost every sequence of the past observations, conditional non-expansivity and membership in \(\mathfrak U(\xi)\) follow.
\qed

%
\subsection{Proof of Theorem \ref{thm:Dobrushin_Doeblin}}\label{ap:proof_thm:Dobrushin_Doeblin}

For any $x'\in\mathcal X$, the triangle inequality in the \(\Sigma^{-1}\)-norm and the boundedness of \(h\) imply
\(
\|x' - h(x)\|_{\Sigma^{-1}}
\le \|x'\|_{\Sigma^{-1}} + s_\Sigma.
\)
Since \(\psi\) is non-decreasing and \(\phi\) is non-increasing, it follows that
\[
p(x, x')
\ge \phi\!\big(\psi(\|x'\|_{\Sigma^{-1}} + s_\Sigma)\big).
\]
Define
\(
q(x'):=\phi\!\big(\psi(\|x'\|_{\Sigma^{-1}} + s_\Sigma)\big).
\)
Then
\(
0\le q(x')\le p(x, x'),
\; \forall x,x'\in\mathcal X.
\)
Now define the finite measure
\[
\nu(A):=\int_A q(x')\,\mathrm dx',
\quad A\in\mathcal B(\mathcal X),
\]
and let \(\kappa:=\nu(\mathcal X)\). Since \(q(x')\le p(x,x')\) for every \(x'\), and \(\int_{\mX}p(x,x')\,\mathrm dx'=1\) for every \(x\), we have \(0<\kappa\le 1\). Define the probability measure
\(
\eta(A):=\dfrac{\nu(A)}{\kappa},
\; A\in\mathcal B(\mathcal X).
\)
Then
\[
K(x,A)
= \int_A p(x, x')\,\mathrm dx'
\ge \int_A q(x')\,\mathrm dx'
= \nu(A)
= \kappa\,\eta(A).
\]
This uniform minorization yields $\|K\| \le 1-\kappa$; see, e.g., \cite[Lemma~4.3.13, p.~98]{Cappe05}.
\qed

%
\subsection{Proof of Proposition \ref{prop:epsilon}}\label{ap:Proof_Proposition}

For every \(s\geq s_\Sigma\), the argument of
Theorem~\ref{thm:Dobrushin_Doeblin} gives the minorization constant
\(\kappa(s)\) and the contraction bound \(\|K\|\leq1-\kappa(s)\).
The change of variables \(x=\Sigma^{1/2}z\), followed by polar coordinates
\(r=\|z\|\), yields
\[
\kappa(s)=S_\Sigma\int_0^\infty
\phi\bigl(\psi(r+s)\bigr)\,r^{d_x-1}\,\mathrm dr,
\]
where \(S_\Sigma\) is defined in Proposition~\ref{prop:epsilon}.

To justify Remark~\ref{rem:radial_minorization_monotonicity}, write
\(\phi(\psi(r))=C f_0(r)\), where the non-increasing radial shape \(f_0\)
is fixed as \(\Sigma\) varies. Normalization of the transition density gives
\[
1=S_\Sigma C\int_0^\infty r^{d_x-1}f_0(r)\,\mathrm dr
\quad \text{and} \quad
\kappa(s)=
\frac{\int_0^\infty r^{d_x-1}f_0(r+s)\,\mathrm dr}
{\int_0^\infty r^{d_x-1}f_0(r)\,\mathrm dr}.
\]
Thus the determinant factor in \(S_\Sigma\) cancels with the normalization
constant \(C\), and \(\kappa(s)\) is independent of \(\Sigma\) for fixed
\(s\) and \(f_0\). Since \(f_0\) is non-increasing, so is \(\kappa(s)\)
as a function of \(s\). Substituting
\(s=b_\Sigma=M/\sqrt{\lambda_{\min}(\Sigma)}\) proves the stated
monotonicity in \(\lambda_{\min}(\Sigma)\) and \(M\).

\noindent\textit{Case 1: Exponential radial profiles.}
Let \(\phi(v)=\mathrm C\,e^{-v}\) and \(\psi(r)=ar^\beta\), where \(\mathrm C,a,\beta>0\). Then the function
\(
\phi\big(\psi(r+s)\big)
= \mathrm C\exp\!\big(-a(r+s)^\beta\big),
\)
and therefore
\[
\kappa(s)=S_\Sigma\,\mathrm C
\int_0^\infty e^{-a(r+s)^\beta}\,r^{d_x-1}\,\mathrm dr.
\]
Introducing the shift \(u=r+s\), we obtain
\(\displaystyle
S_\Sigma\,\mathrm C
\int_{s}^\infty e^{-au^\beta}(u-s)^{d_x-1}\,\mathrm du.
\)
Since \(d_x\in\mathbb N\), expanding \((u-s)^{d_x-1}\) via the binomial theorem yields
\[
S_\Sigma\,\mathrm C
\sum_{k=0}^{d_x-1}
\binom{d_x-1}{k}
(-s)^{d_x-1-k}
\int_{s}^\infty u^k e^{-au^\beta}\,\mathrm du.
\]
The remaining integral admits the closed form
\[
\int_{s}^{\infty}
u^k e^{-a u^\beta}\,\mathrm du
=
\frac{1}{\beta}\,
a^{-(k+1)/\beta}\,
\Gamma\!\left(
\frac{k+1}{\beta},\, as^\beta
\right),
\]
where \(\Gamma(\cdot,\cdot)\) denotes the upper incomplete gamma function. Substituting this expression yields \eqref{eq:epsilon_Exp}.

\noindent\textit{Case 2: Polynomial-tail profiles.}
Let \(\phi(v)=\mathrm C\,v^{-\alpha}\) and \(\psi(r)=1+ar^2\), where \(a>0\) and \(\alpha>d_x/2\). Then the function
\(
\phi\big(\psi(r+s)\big)
=
\mathrm C\,(1 + a(r+s)^2)^{-\alpha}.
\)
After the radial change of variables, the integral becomes
\[
\kappa(s)=S_\Sigma\,\mathrm C
\int_0^\infty
(1 + a(r+s)^2)^{-\alpha} r^{d_x-1}\,\mathrm dr.
\]
Introducing the shift \(u=r+s\) yields
\(\displaystyle
S_\Sigma\,\mathrm C
\int_{s}^{\infty}
(1 + a u^2)^{-\alpha} (u-s)^{d_x-1}\,\mathrm du.
\)
By the binomial theorem, 
\(\displaystyle
(u-s)^{d_x-1}
=
\sum_{k=0}^{d_x-1}
\binom{d_x-1}{k}
(-s)^{d_x-1-k} u^k,
\)
so that
\[
S_\Sigma\,\mathrm C
\sum_{k=0}^{d_x-1}
\binom{d_x-1}{k}
(-s)^{d_x-1-k}
\int_{s}^{\infty} u^k (1 + a u^2)^{-\alpha}\,\mathrm du.
\]
For \(k=0,\ldots,d_x-1\), the substitution
\(z=au^2/(1+au^2)\), with \(z_s=as^2/(1+as^2)\), gives
\[
\begin{aligned}
I_k(s)
&:=\int_{s}^{\infty}u^k(1+au^2)^{-\alpha}\,\mathrm du
=\frac12 a^{-(k+1)/2}
\int_{z_s}^{1}z^{(k-1)/2}(1-z)^{\alpha-(k+3)/2}\,\mathrm dz\\
&=\frac12 a^{-(k+1)/2}
\left[
B\!\left(\frac{k+1}{2},\alpha-\frac{k+1}{2}\right)
-B_{z_s}\!\left(\frac{k+1}{2},\alpha-\frac{k+1}{2}\right)
\right].
\end{aligned}
\]
Both beta parameters are positive for every \(k=0,\ldots,d_x-1\) when
\(\alpha>d_x/2\), as required for normalization of the radial profile.
This is \eqref{eq-aux-prop5.7}; substituting into the preceding sum yields
\eqref{eq:epsilon_Pol}, with the factor \(1/2\) already included in \(I_k(s)\).
\qed


\renewcommand{\appendixname}{}
\bibliographystyle{imsart-number.bst} 
\bibliography{bibliography}       

%
%

\end{document}